\documentclass[reqno, 11pt]{amsart}

\usepackage{amsfonts, amssymb, amsmath, amsthm, mathtools, floatflt, palatino, euler, epic, eepic, microtype, lineno, multirow, multicol, comment, cases, graphicx, float}
\usepackage[colorlinks, linkcolor=darkred, citecolor=darkblue, pagebackref=true, pdftex]{hyperref}
\usepackage[usenames]{xcolor}
\usepackage[linesnumbered, ruled]{algorithm2e} 
\usepackage[normalem]{ulem}
\usepackage[all]{xy}
\usepackage[margin=1 in]{geometry}
\modulolinenumbers[1]

\newcommand{\CC}{\mathfrak{C}}

\newcommand{\GG}{\mathfrak{G}}
\newcommand{\II}{\mathfrak{I}}
\newcommand{\SIC}{\mathfrak{I}_S}

\newcommand{\Lk}{\mathrm{Lk}}

\newtheorem{theorem}{Theorem}
\newtheorem{lemma}[theorem]{Lemma}

\newtheorem{definition}[theorem]{Definition}

\newtheorem{conjecture}[theorem]{Conjecture}
\newtheorem{remark}[theorem]{Remark}

\newtheorem{example}[theorem]{Example}
\newtheorem*{theorem*}{Theorem}
\newtheorem*{cor*}{Corollary}
\newtheorem*{prop*}{Proposition}
\newtheorem*{conjecture*}{Conjecture}

\def\kwA{Contractible transformations}
\def\kwB{collapsible graph}
\def\kwC{Vietoris-Rips complex}
\def\kwD{persistent homology}

\definecolor{darkblue}{rgb}{0,0,0.7}
\definecolor{darkred}{rgb}{0.7,0,0}

\makeatletter
\def\BState{\State\hskip-\ALG@thistlm}
\makeatother

\begin{document}

\title{Collapsibility and homological properties of $\II$-contractible transformations}

\author[J. F. Espinoza]{Jes\'us F. Espinoza}
\email{jesusfrancisco.espinoza@unison.mx}

\author[M. E. Fr\'ias-Armenta]{Mart\'in Eduardo Fr\'ias-Armenta}
\email{eduardo.frias@unison.mx}

\author[H. A. Hern\'andez-Hern\'andez]{H\'ector Alfredo Hern\'andez-Hern\'andez}
\email{hector.hernandez@unison.mx}
\address{Departamento de Matem\'aticas, Universidad de Sonora, M\'exico}

\begin{abstract}
The family of $\II$-contractible graphs and contractible transformations was defined by A. Ivashchenko in the mid-90's. In this paper we study the collapsibility and homological properties of the clique complex associated to $\II$-contractible graphs. We show that for any graph in a special subfamily of the $\II$-contractible graphs (the strong $\II$-contractible ones) its clique complex is collapsible. Moreover, we present an algorithm that allows us to verify if any graph is strong $\II$-contractible, as well as an algorithm to delete those vertices whose open neighborhood is also strong $\II$-contractible. Finally, we show how to use these algorithms to compute the persistent homology of an arbitrary Vietoris-Rips complex for applications in topological data analysis.
\end{abstract}


\maketitle

\section{Introduction} 
\subsection{Preliminaries}
In this work will only consider finite and simple graphs, let $\GG$ denote the set of such graphs.

An important kind of subgraph for the present work is given by the open neighborhood (or just neighborhood) of any vertex $v$ in $G$, denoted by $N_G(v)$ and defined as the induced subgraph in $G$ by the set of all adjacent vertices to $v$. In the same way, the {\it common neighborhood} of two vertices $v$ and $w$ of $G$, is the subgraph $N_G(v,w)$ induced by the set of vertices which are adjacent to $u$ and $v$.

\begin{figure}[htb]
\begin{minipage}[b]{.49\linewidth}
\centering \includegraphics[width=0.8\linewidth]{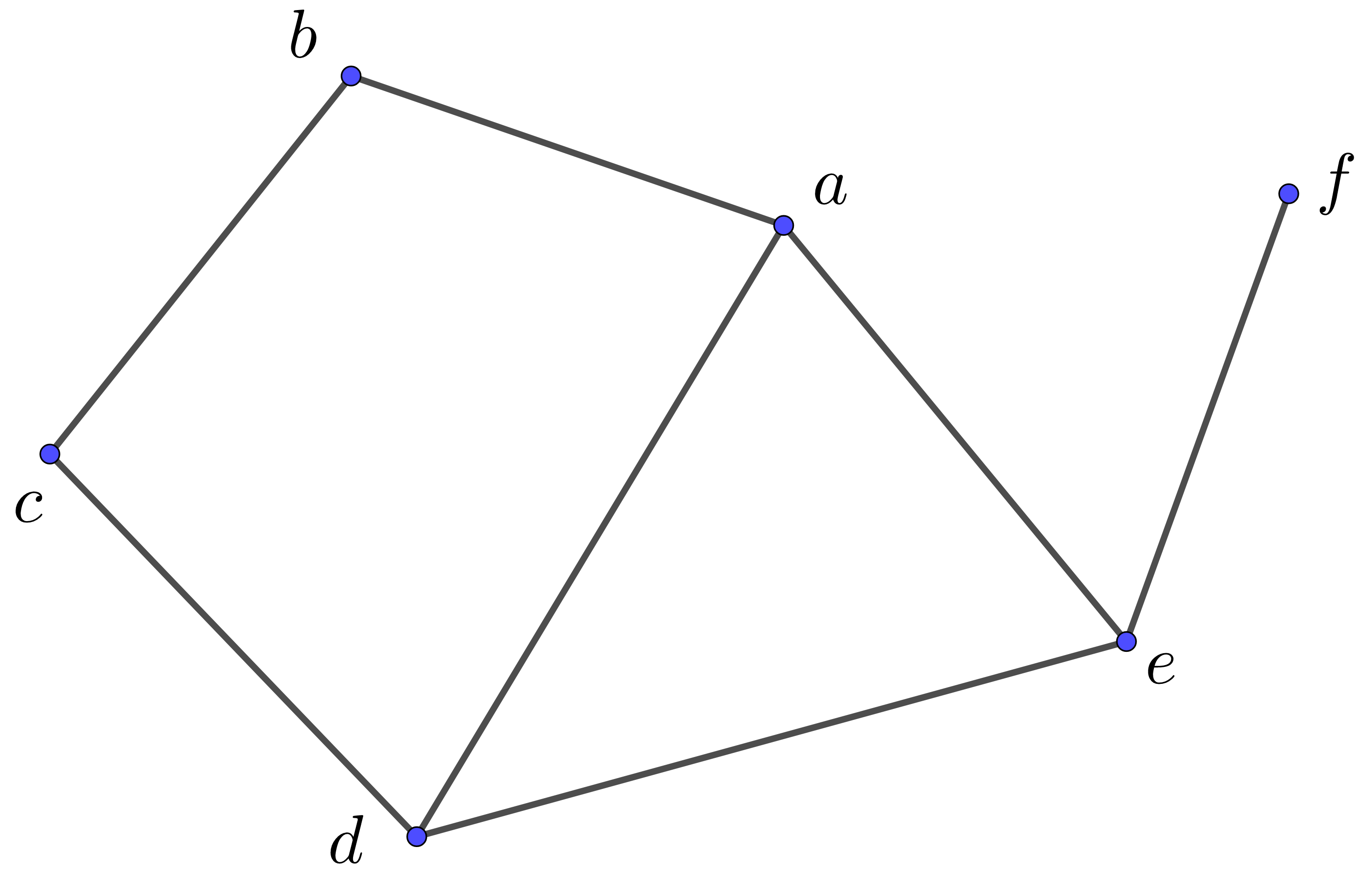}
\end{minipage}
\begin{minipage}[b]{.49\linewidth}
\centering \includegraphics[width=0.8\linewidth]{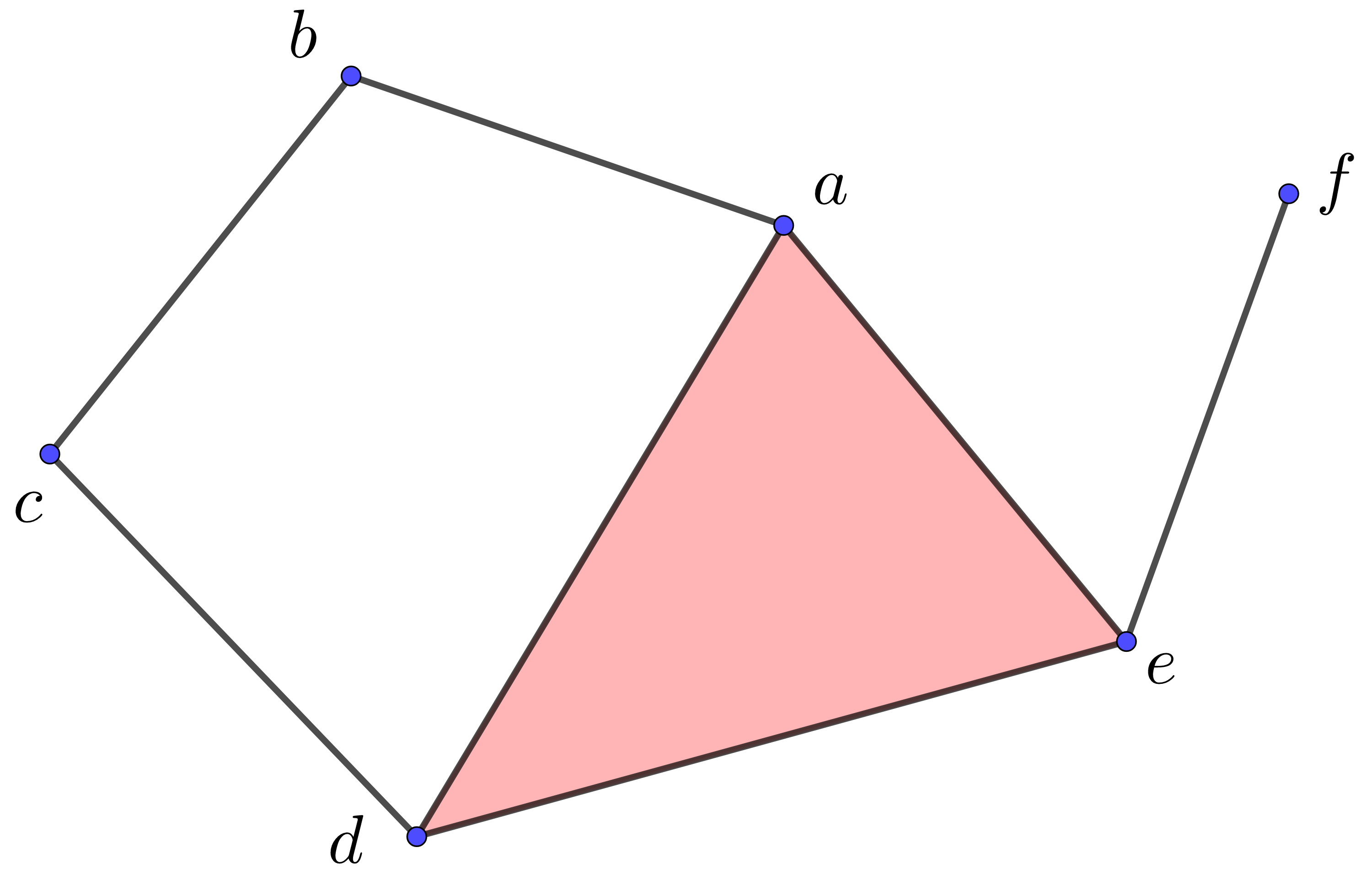}
\end{minipage}
\caption{A graph $G$ and its clique complex $\Delta (G)$.}\label{fig:1}
\end{figure}

For instance, the open neighborhood of the vertex $a$ in graph $G$ shown to the left in Figure \ref{fig:1} is the subgraph $N_G(a)$ with vertices $V(N_G(a)) = \{b, d, e\}$ and edges $ E(N_G(a)) = \{\{d, e\}\}$. Also, the common neighborhood of vertices $a,c \in V(G)$ is the subgraph $N_G(a,c)$ with vertices $\{b,d\}$ and no edges. For basic concepts in graph theory you can see \cite{harary2019graph}.

\textit{Simplicial complexes} can be thought of as the higher-dimensional analogues of graphs (see \cite[Chapters 2-3]{Kozlov:2007} for an excellent introduction to abstract and geometric simplicial complexes, as well as its (co)homology groups). More precisely, a finite (abstract) simplicial complex is a finite set $V$ together with a collection $\Delta \subseteq 2^V$ closed under inclusion, i.e., if $\tau \in \Delta$ and $\sigma \subseteq \tau$, then $\sigma \in \Delta$. An element $\sigma \in \Delta$ is called \textit{simplex}; if additionally $\sigma \in V$, then $\sigma$ is called a \textit{vertex} of $\Delta$. For two simplices $\sigma,\tau \in \Delta$ such that $\sigma \subset \tau$, we say that $\sigma$ is a \emph{face} of $\tau$ and, if there is no longer simplices for which $\tau$ is a face, then $\tau$ is called a \emph{maximal face} of $\Delta$. The dimension of the simplex $\sigma$ is defined as $\dim(\sigma) = |\sigma|-1$. Therefore, the $0$-simplices are just the vertices of $\Delta$, the $1$-simplices are the edges, $2$-simplices are the triangles, and so on. We denote by $\Delta^n$ the complete simplicial structure associated with the set $\{0,1,\ldots, n\}$. 

Given a simplicial complex $\Delta$, a couple of simplices $\sigma, \tau \in \Delta$ are called a \textit{free pair} if the following two conditions are satisfied: (1) $\sigma \subsetneq \tau$ and, (2) $\tau$ is a maximal face of $\Delta$ with respect to inclusion, and no other maximal face of $\Delta$ contains $\sigma$. For instance, the simplicial complex $\Delta = \{\emptyset, \{a\}, \{b\}, \{c\}, \{a,b\}, \{b,c\}\}$ has exactly two free pairs: $(\{a\}, \{a,b\})$ and $(\{c\}, \{b,c\})$.

A \textit{simplicial collapse} of $\Delta$ in $\tilde{\Delta}$ is obtained by removal of all simplices $\gamma\in \Delta$ such that $\sigma \subseteq \gamma \subseteq \tau$ provided that $(\sigma, \tau)$ is a free pair, we will write $\Delta \searrow \tilde{\Delta}$. Additionally, if $\dim(\tau) = \dim(\sigma) + 1$, then $\Delta \searrow \tilde{\Delta}$ is called an \textit{elementary simplicial collapse}. It is not hard to see that any simplicial collapse can be realized by elementary ones (cf. \cite[Sec. III.4]{Edelsbrunner:08}). In any case, the removal of the free pair $(\sigma,\tau)$ of $\Delta$ is denoted by $C(\Delta; (\sigma,\tau)) = \Delta-(\sigma,\tau)$.

For $n\geq 1$ if there are $\Delta_1, \Delta_2, \ldots , \Delta_n$ simplicial complexes such that $\Delta_1 \searrow \Delta_2 \searrow \cdots \searrow \Delta_n$, we say that $\Delta_1$ is collapsible to $\Delta_n$. We denote this also by $\Delta_1 \searrow \Delta_n$. In particular, if $\Delta$ is collapsible to $\Delta^0$ we just say that $\Delta$ is collapsible.  We write \[C(\Delta; ((\sigma_1,\tau_1),(\sigma_2,\tau_2), \ldots, (\sigma_k,\tau_k))) = C(C(\Delta; (\sigma_1,\tau_1)); ((\sigma_2,\tau_2), \ldots, (\sigma_k,\tau_k))),\] for a sequence of free pairs collapses.

The inverse process of a simplicial collapse is called a simplicial expansion, i.e., if $\Delta$ is a simplicial complex and $(\sigma, \tau)$ is a free pair, then we say that $\Delta - (\sigma, \tau)$ expands to $\Delta$, which is equivalent to say $\Delta$ collapses to $\Delta - (\sigma, \tau)$. We say that a simplicial complex $\Delta$ is \textit{contractible} if it can be transformed to $\Delta^0$ by a series of collapses and expansions. Of course, any collapsible simplicial complex is also contractible but the converse is not true, take as example the \emph{dunce hat} (see \cite{Zeeman:1963}) or the \emph{Bing's house} (see \cite{Cohen:1973}).

Let $G$ be a graph, the \textit{clique complex} $\Delta(G)$ (see \cite[Sec. 9.1]{Kozlov:2007}) is defined as the simplicial complex obtained by considering any \textit{complete subgraph} $K_n \subset G$, i.e., a graph on $n$ vertices and all possible edges (according the notation in \cite{harary2019graph}), as an $(n-1)$-simplex,
\[\Delta(G) := \{ \sigma \subset G \mid \sigma \mbox{ is a complete subgraph of } G \}.\]

The picture to the right in Figure \ref{fig:1} shows the clique complex $\Delta(G)$ of the graph $G$ shown to the left. The $1$-skeleton of $\Delta(G)$ can be identified with $G$ itself.

\begin{definition}
Let $G$ be a graph. If $\Delta(G)$ is collapsible, then the graph $G$ is called collapsible. We denote the family of all collapsible graphs by $\CC$.

Moreover, if $\Delta(G)$ is a contractible simplicial complex, then $G$ is called contractible.
\end{definition}

For example, any complete graph $K_n$ is collapsible (and also contractible): $\Delta(K_n) \searrow \Delta^0$. Also, the graph $G$ shown in Figure \ref{fig:1} corresponds to a non-collapsible graph, because once all the free pairs are deleted in the clique complex $\Delta(G)$, a square with vertices $a$, $b$, $c$ and $d$, is obtained, which no longer has free pairs, so it is impossible to continue the collapsing process until you get a point. Such collapses can consists in deleting first the free pair $(\{f\}, \{e, f\})$ and then the now free pair $(\{e\}, \{a, d, e\})$ in $\Delta(G)$. On the other hand, Example \ref{example-octahedron} shows a graph (Figure \ref{figure:octahedron2}) that is collapsible.

In \cite{Ivashchenko:homology} the following family of graphs was defined, and its elements are called $\II$-contractible graphs.

\begin{definition}\label{IF} We denote by $\II$ the family {\bf $\II$-contractible graphs}, defined as follows:
\begin{enumerate}
\item The trivial graph $K_1$ is in $\II$.
\item We add a graph $G'$ to $\II$ through the next four operations:

\begin{enumerate}
\item[(I1)] Deleting a vertex $v$. We add the graph $G'$ to $\II$, if it is obtained by deleting a vertex $v$ of a graph $G\in\II$ when $N_G(v)\in \II$.
\item[(I2)] Gluing  a vertex $v$.  We add the graph $G'$ to $\II$, if it is obtained by adding a vertex $v$ to a graph $G\in \II$ in such a way that $N_G(v)\in \II$.
\item[(I3)] Deleting an edge $\{v_1,v_2\}$. We add the graph $G'$ to $\II$, if it is obtained by deleting an edge $\{v_1,v_2\}$ of a graph $G\in\II$ when $N_G(v_1,v_2) \in \II$.
\item[(I4)] Gluing an edge $\{v_1,v_2\}$.  We add the graph $G'$ to $\II$, if it is obtained by adding an edge $\{v_1,v_2\}$ to a graph $G\in\II$ when $N_G(v_1,v_2) \in \II$.
\end{enumerate} 
\end{enumerate}
We call $G$ a $\II$-contractible graph if $G$ belongs to $\II$. A vertex $v \in V(G)$ (resp. edge $\{v_1, v_2\} \in E(G)$) is called $\II$-contractible if $N_G(v)$ (resp. $N_G(v_1, v_2)$) is $\II$-contractible.
\end{definition}

For any vertex $v \in G$, we write $I(G; v) = G - v$ when  $N_G(v)$ is in the family $\II$. Let $S = (v_1, \ldots, v_k) \subset V(G)$ be a nonempty ordered subset of vertices of the graph $G=(V, E)$. We call $S$ a \textit{sequence of $\II$-contractible vertices} of $G$ if $v_1$ is an $\II$-contractible vertex of $G_1=G$ and, for $k\geq 2$, $v_i$ is $\II$-contractible in $G_{i} = I(G_{i-1}; v_{i-1})$ for any $2 \leq i \leq k$. If $S =(v_1, \ldots, v_k)$ is a sequence of $\II$-contractible vertices of the graph $G$, we define $I(G; S) = I( I(G;v_1); (v_2, \ldots, v_k))$.

A remarkable subfamily of the $\II$-contractible graphs is the family of graphs $\SIC$ obtained from $K_1 \in \SIC$ by gluing a vertex $v$ to any graph $G\in \SIC$ in a subgraph $G' \subset G$ such that $G' \in \SIC$, or gluing an edge $\{v_1,v_2\}$ between two nonadjacent vertices $v_1, v_2 \in G$ when $N_G(v_1,v_2) \in \SIC$. The family $\SIC$ is the main subject of study in this paper. Our goals are, firstly to identify the relationship of $\SIC$ with respect to other more studied families of graphs, and secondly to take advantage of its recursive construction and nice homological properties in order to construct algorithms to compute persistent homology, as an application in Topological Data Analysis (see Section \ref{Appli}).

\begin{definition}\label{TSI} Let $\SIC \subset \GG$ be the family of graphs defined by:
\begin{enumerate}
\item The trivial graph $K_1$ is in $\SIC$.
\item We add a graph $G'$ to $\SIC$ through the next two operations:
\begin{enumerate}
\item[(I2)] Gluing  a vertex $v$.  We add the graph $G'$ to $\SIC$, if it is obtained by adding a vertex $v$ to a graph $G\in \SIC$ in such way that $N_G(v)\in \SIC$.
\item[(I4)] Gluing an edge $\{v_1,v_2\}$.  We add the graph $G'$ to $\SIC$, if it is obtained by adding an edge $\{v_1, v_2\}$ to a graph $G \in \SIC$ when $N_G(v_1,v_2) \in \SIC$.
\end{enumerate} 
\end{enumerate}
If $G$ belongs to $\SIC$, then $G$ is called a strong $\II$-contractible graph. 
\end{definition}

Any complete graph $K_n$ and any tree are examples of strong $\II$-contractible graphs. In \cite{chen} it was proven that the Bing's house is $\II$-contractible but not strong $\II$-contractible. On the other hand, the graph in Figure \ref{fig:1} is neither $\II$-contractible nor strong $\II$-contractible.

\subsection{Our contributions}
In Section \ref{Collpasible-graphs} we prove the following statement.

\newtheorem*{inclusionIC}{Theorem \ref{inclusionIC}}
\begin{inclusionIC}
If $G$ is a strong $\II$-contractible graph, then it is a collapsible graph.
\end{inclusionIC}

We prove Theorem \ref{inclusionIC} by double induction on the number of vertices and edges of the graph and the dimension of its clique complex. Additionally, we conjecture that both families agree, $\SIC = \CC$, but we do not have a complete proof; however, we have extensive computational evidence that support such conjecture. 
\begin{conjecture}\label{SIV=C}
The family of strong $\II$-contractible graphs is the same as the family of collapsible graphs.
\end{conjecture}
Moreover, we conjecture that such relationship can be stated in the family of $\II$-contractible graphs.
\begin{conjecture}
The family of $\II$-contractible graphs is the same as the family of contractible graphs.
\end{conjecture}

In the conclusion of \cite{Lofano}, the authors presented a simplicial complex $C^8_3$ that is contractible but is not collapsible or strong-expandable. We think that the $1$-skeleton of the barycentric subdivision of $C^8_3$ could be a counterexample to the claim in \textit{Theorem 3.7} of \cite{Ivashchenko:contractible}, one of the just three results that was not refuted in \cite{Frias}.

We begin Section \ref{CACG} with an example for showing that Theorem \ref{inclusionIC} is not trivially reversible. Given the difficulty to prove our conjecture $(\CC \subset \SIC)$, we support it with the help of a computer. Thus, we have verified, with graphs of up to 9 vertices, the conjecture that both graph families coincide. To do that, first we generated a limited bank of graphs that were representative of their isomorphism class. This set started with the graph of a single point $K_1$, then we generated new graphs using two operations: (1) we connected a new vertex through an edge and (2) we added a new edge. Then we verified it was not in the bank, and we added it to the bank. Once this had been achieved, we applied algorithms to determine if they belonged to both $\CC$ (exhaustive searching of free pairs) and $\SIC$ (Algorithm \ref{Alg:contractible.graph}) families or not. 

The use of graphs and simplicial structures have recently had a huge boom in applications for data analysis. Such applications belong to the area of Topological Data Analysis (TDA), whose goal is to identify geometric properties of the underlying space in which a (finite) point cloud belongs; these points are obtained through experiments, simulations, databases, recordings, etc. The \textit{simplicial homology} is a fundamental tool to perform such analysis, more precisely: \textit{simplicial persistent homology}. The key reason for that, is because the distinct groups of homology encode, in a certain sense, the shape of a geometric structure.

The workflow of the TDA is as follows: given a collection of data in which there is a notion of distance, a filtered simplicial structure is constructed, then the persistence homology groups are calculated and represented by a barcode or persistence diagram. Finally, based on this information, the shape of the data is inferred, indicating (up to topological noise) the number of connected components and the number of higher-dimensional holes (persistent Betti numbers).

For an introduction to the fundamentals of TDA and an excellent panoramic survey about concepts, applications and computational tools, see \cite{Chazal2021}.

In Section \ref{Appli} we show an application of the $\II$-contractible transformations to computation of the homology of a special kind of simplicial complexes: given a non-negative real number $\varepsilon$ (\textit{proximity parameter}) and $N$ a (finite) \textit{point cloud} in a metric space, the Vietoris-Rips complex $\mathrm{VR}(N; \varepsilon)$ is the simplicial complex with a $\sigma \subset N$ being a simplex if, and only if, the distance between any two points in $\sigma$ is less than or equal to $\varepsilon$.

One of the last results of this paper proves that computing the persistent homology of the Vietoris-Rips complex is equivalent to computing the homology groups of the corresponding underlying graphs reduced by contractible transformations. This result is proven in Section \ref{Appli}.

\newtheorem*{persistence-preserved}{Theorem \ref{persistence-preserved}}
\begin{persistence-preserved}
Let $\mathrm{VR}_0 \subset \mathrm{VR}_1 \subset \cdots \subset \mathrm{VR}_m$ be the Vietoris-Rips filtration of any finite point cloud, and let $H_p^{i,j}$ be the $(i,j)$-persistent $p$-homology group of the filtration.  If $G_i := \mathrm{VR}_i^{(1)}$ given by the 1-skeleton of the complex $\mathrm{VR}_i$ and $S_i$ is a sequence of $\II_S$-contractible vertices of $G_i$, for all $1 \leq i \leq m$, then
\begin{equation*}
H_p^{i,j} \cong \mathrm{Im}(\iota_* : H_p(\Delta(I(G_i; S_i))) \to H_p(\Delta(I(G_j; S_j)))).
\end{equation*}
\end{persistence-preserved}

What is remarkable about this result is that even when such a collection of (strong $\II$-contractible transformed) graphs is not a filtration, there is an induced homomorphism at the level of the homology groups that allows us to recover the persistent homology of the Vietoris-Rips complex. This approach is a topic of interest, as is shown in references {\cite{Boissonnat:2018}}-{\cite{Boissonnat:2019}} published a posteriori of a preprint of this paper, first available online in {\cite{EFH}}. 

We begin our study in Section \ref{GH-CT} by defining the graph homology groups as were originally presented in \cite{Ivashchenko:homology} and \cite{Ivashchenko:contractible}. In the section we show the existence of an induced homomorphism between the homology groups of the clique complexes of two graphs, induced by the image of contractible transformations (Lemma \ref{diagram}), a key result that permits the extension of the full theory of $\II$-contractible transformations to study the persistent homology.

\section{Collapsibility properties of \texorpdfstring{$\II$}{I}-contractible transformations}\label{Collpasible-graphs}

We start the section by introducing some notation about simplicial complexes (see \cite{Kozlov:2007}). Let $\Delta$ be a simplicial complex and let $\alpha$ be a simplex of $\Delta$. The \emph{link} of $\alpha$ is the simplicial subcomplex of $\Delta$, denoted by $\mathrm{Lk}(\alpha; \Delta)$ and defined by
\[\mathrm{Lk}(\alpha; \Delta) := \{\gamma \in \Delta \mid \gamma \cap \alpha = \emptyset \mbox{ and } \alpha \cup \tau \in \Delta\}.\]
The \emph{deletion} of $\alpha$ is the simplicial subcomplex of $\Delta$, denoted by $\mathrm{del}_\Delta(\alpha)$ and defined by
\[\mathrm{del}_\Delta(\alpha) := \{\gamma \in \Delta \mid \gamma \nsupseteq \alpha\}.\]
On the other hand, if $\Delta_1$ and $\Delta_2$ are two simplicial complexes whose vertices are indexed by disjoint sets, then the \emph{join} of $\Delta_1$ and $\Delta_2$ is the simplicial complex $\Delta_1 * \Delta_2$ whose set of vertices is $V(\Delta_1) \cup V(\Delta_2)$ and the set of simplices is given by
\[\Delta_1 * \Delta_2 := \{\sigma \subseteq V(\Delta_1) \cup V(\Delta_2) \mid \sigma \cap V(\Delta_1) \in \Delta_1 \mbox{ and } \sigma \cap V(\Delta_2) \in \Delta_2\}.\]
For two simplices $\sigma, \tau \in \Delta$ such that $\sigma \cap \tau = \emptyset$, its join $\sigma * \tau$ can be thought as the simplex with set of vertices $\sigma \cup \tau$, even if $\sigma * \tau \not\in \Delta$.

The next theorem proves that every strong $\II$-contractible graph is also a collapsible graph.
\begin{theorem}\label{inclusionIC}
If $G$ is a strong $\II$-contractible graph, then it is a collapsible graph.
\end{theorem}

\begin{proof}
Let $n = \dim(\Delta(G))$ be the dimension of the clique complex $\Delta(G)$, and let  $k = |V|+|E|$ denote the sum of its number of vertices and its number of edges. 

We will prove the claim by induction on $n$ and $k$. The claim is true for $n=0$ or if $k=1$, in which case $\Delta(G) = \Delta^0$.

Now suppose that the claim is true for all $m<n$ and for $j<k$ if $m=n$.
 
Let $\alpha\in G$ be a vertex or edge such that $N_G(\alpha) \in \SIC$, then it follows from the induction hypothesis over $n$ that 
\begin{linenomath*}
\[ \Delta(N_G(\alpha)) = \Lk(\alpha;\Delta(G)) \searrow \Delta^0. \]
\end{linenomath*}

Therefore, there exists a sequence of elementary collapses and their corresponding free pairs $(\sigma_i, \tau_i)$, $\dim(\tau_i) = \dim(\sigma_i) + 1$, in $\Lk(\alpha;\Delta(G))$ such that 
\begin{linenomath*}
\[ \Delta^0 = (\cdots ((\Delta(N_G(\alpha))-(\sigma_1,\tau_1))-(\sigma_2,\tau_2))\cdots ). \]
\end{linenomath*}

On the other hand, each couple $(\alpha*\sigma_i, \alpha*\tau_i)$ is also a free pair of the simplicial subcomplex $\alpha*\Lk(\alpha; \Delta(G))$ for each $i$. In fact, they are free pair in the clique complex $\Delta(G)$. Thus, 
\begin{linenomath*}
\[ \alpha*\Lk(\alpha;\Delta(G)) \searrow \alpha*\Delta^0 \searrow \Delta^0 \]
\end{linenomath*}
and consequently, $\Delta(G) \searrow \Delta(G-\alpha)$. Now $|V(G-\alpha)|+|E(G-\alpha)|\leq k-1$ and by the induction hypothesis over $k$, we have $\Delta(G-\alpha) \searrow \Delta^0$, which induces the desired simplicial collapses $\Delta(G) \searrow \Delta(G-\alpha) \searrow \Delta^0$.
\end{proof}

On the other hand, let us recall that the class of collapsible simplicial complexes can be recursively defined as:
\begin{itemize}
    \item The empty complex and the 0-simplex are collapsible.
    \item If $\Delta$ contains a nonempty face $\sigma$ such that $\mathrm{del}_{\Delta}(\sigma)$ and $\mathrm{Lk}(\sigma; \Delta)$ are collapsible, then $\Delta$ is collapsible.
\end{itemize}

If we restrict the face removing in such definition to 0-cells or 1-cells, then we get a subfamily of collapsible simplicial complexes, which we will refer to as \textit{1-sk collapsible simplicial complexes}, and this subfamily can also be defined recursively:
\begin{itemize}
    \item The empty complex and the 0-simplex are 1-sk collapsible.
    \item If the simplicial complex $\Delta$ contains a nonempty 0-cell or 1-cell $\sigma$ such that $\mathrm{del}_\Delta(\sigma)$ and $\mathrm{Lk}(\sigma; \Delta)$ are 1-sk collapsible, then $\Delta$ is 1-sk collapsible. 
\end{itemize}

Now, we can define a graph $G$ to be \textit{1-sk collapsible} if its clique complex $\Delta(G)$ is 1-sk collapsible. This yields to the following lemma.

\begin{lemma}
Let $G$ be a graph. Then $G \in \SIC$ if, and only if, $G$ is 1-sk collapsible.
\end{lemma}
\begin{proof}
For any graph $G$, let $k(G) = |V(G)|+|E(G)|$ denote the sum of its number of vertices and its number of edges. 

Let $G \in \SIC$ be a strong $\II$-contractible graph. We will prove that $\Delta(G)$ is 1-sk collapsible. Let $\alpha\in G$ be a vertex or edge such that $N_G(\alpha)$ and $G-\alpha$ belongs to $\SIC$, then it follows by induction over $k(G)$ that $\Delta(G)$ is 1-sk collapsible.

On the other hand, if $G$ is such that $\Delta(G)$ is 1-sk collapsible, then there is a 0-cell or 1-cell $\alpha \in \Delta(G)$ such that $\Lk(\alpha;\Delta(G))$ and $\mathrm{del}_{\Delta(G)}(\alpha)$ are 1-sk collapsible, but
\[\Lk(\alpha;\Delta(G)) = \Delta(N_G(\alpha))\]
and 
\[\mathrm{del}_{\Delta(G)}(\alpha) = \Delta(G - \alpha).\]
Then, from the induction hypothesis over $k(G)$, it follows that $N_G(\alpha)$ and $G - \alpha$ are strong $\II$-contractible. Therefore $G \in \SIC$. 
\end{proof}

As we claim in Conjecture \ref{SIV=C}, we conjecture that the family of strong $\II$-contractible graphs is exactly the family of collapsible graphs, i.e., the 1-sk collapsibility of any simplicial clique complex is a sufficient condition for its collapsibility.

\subsection{Computational approach in the study of the contractible graphs}\label{CACG} 
To verify that $\SIC\subseteq\CC$, we can use the sequence of contractible transformations as a guide to make the collapses, as shown in Theorem \ref{inclusionIC}. It is, however, not possible the other way. We now examine the example below.

\begin{example}\label{example-octahedron}
Let $G$ be the graph to the left in Figure \ref{figure:octahedron2}. The clique complex $\Delta (G)$ is given at the right in the picture. 

\begin{figure}[H]
\begin{minipage}[t]{0.49\linewidth}
\centering
	\includegraphics[width=0.6\linewidth]{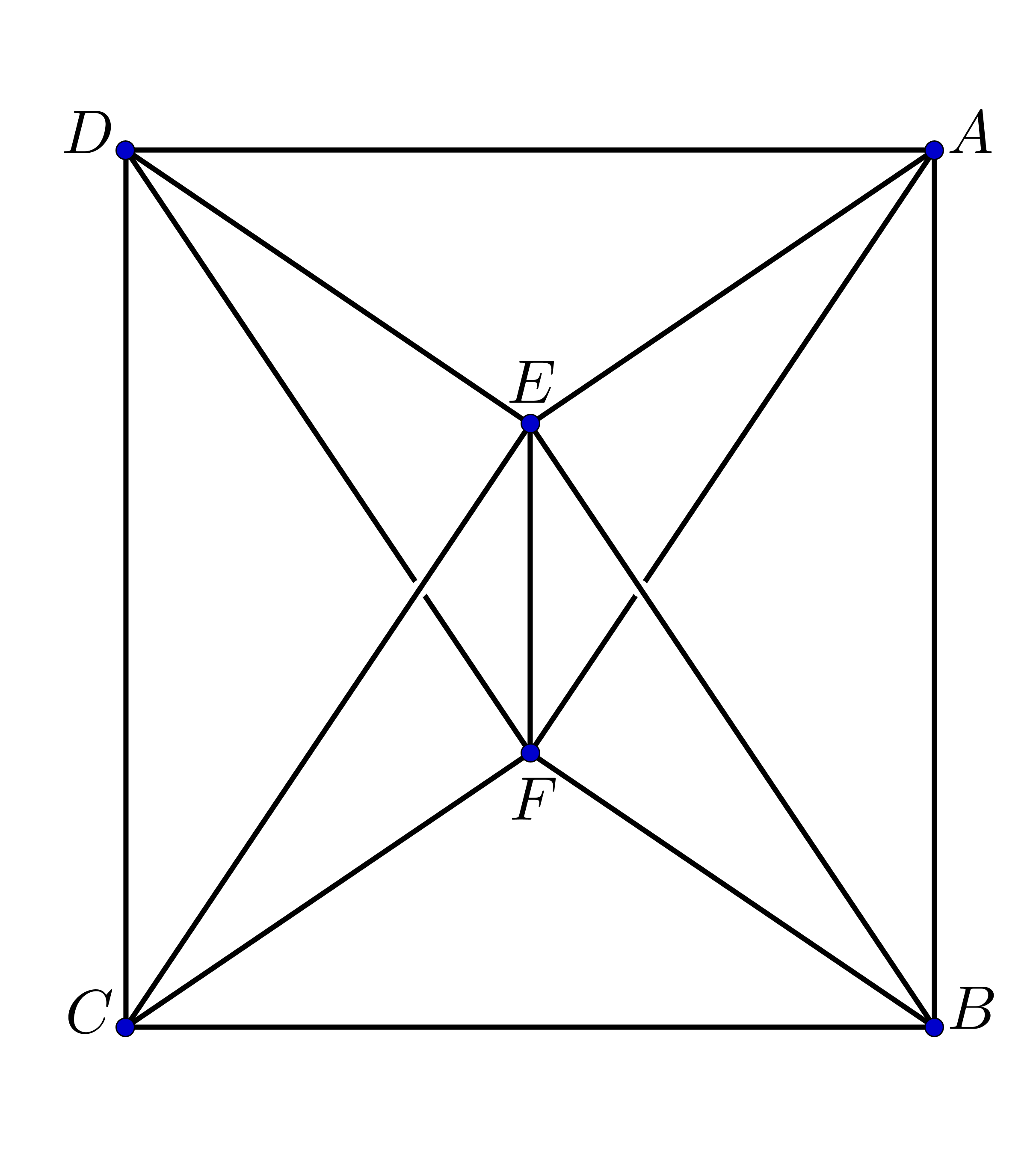}
\end{minipage}
\hfill
\begin{minipage}[t]{0.49\linewidth}
    \includegraphics[width=0.6\linewidth]{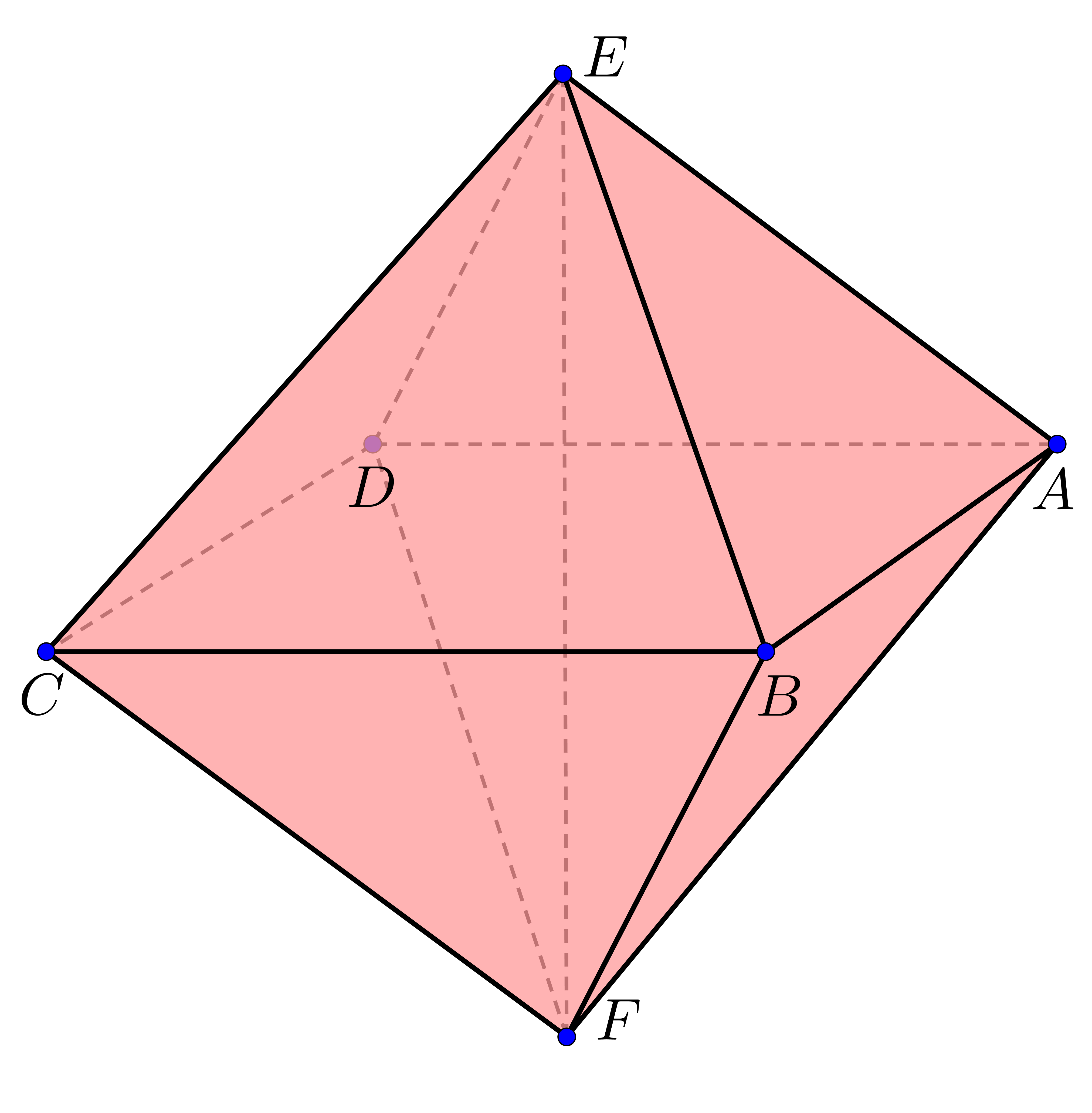}
\end{minipage} 
\caption{The graph $G$ and its clique complex $\Delta (G)$.}
\label{figure:octahedron2}
\end{figure}

Clearly $G \in \SIC$, and $\Delta (G)$ is collapsible, i.e., $G \in \CC$. In order to reverse the arguments in Theorem \ref{inclusionIC}, any removing sequence of free pairs should induce a sequence of contractible transformations. The collapses, however, cannot induce a contractible transformation:
\begin{linenomath*}
\begin{align*}
\Delta_0 := \Delta(G) &\searrow \Delta_1 := C(\Delta_0 ; (\{A,B,E\}, \{A,B,E,F\})) \\
& \searrow  \Delta_2 := C(\Delta_1 ; (\{A,E,F\}, \{A,D,E,F\})) \\
& \searrow  \Delta_3 := C(\Delta_2 ; (\{D,E,F\}, \{C,D,E,F\})) \\
& \searrow  \Delta_4 := C(\Delta_3 ; (\{C,E,F\}, \{B,C,E,F\})) \\
& \searrow  \Delta_5 := C(\Delta_4 ; (\{E,F\}, \{B,E,F\})).
\end{align*}
\end{linenomath*}
In fact, the 1-skeleton of $\Delta_5$ is not a contractible graph, and $N_G(\{E,F\}) \not\in \II$.
\end{example}

A vertex $v \in G$ is \textit{dominated} or \textit{dismantlable} in the graph $G$ if there exists $w \in G$, $w\neq v$, such that $N_G[v] \subseteq N_G[w]$, where $N_G[v] := N_G(v) \cup \{v\}$ denotes the closed neighborhood of the vertex $v$. A graph $G$ is \textit{dismantlable} if there is a sequence $v_1, \ldots, v_n$ of all its vertices such that $v_i$ is dismantlable in $G - \{v_1, \ldots, v_{i-1}\}$. In \cite{Prisner} it was proven that any dismantlable graph is contractible; the definition of a dominated vertex in graphs appears in \cite{Escalante} as a partial order, and the concept of dismantling appears as taking some maximal classes, see \cite[Th. 2]{Escalante} and the comments before it. But it is not until the work in \cite{Nowakowski} that a vertex is defined as dominated or dismantlable exactly as it is now conceived. In \cite{Boulet} this family was extended to $s$-dismantlable graphs. Furthermore, the $s$-moves were defined, and it was proven that $s$-moves on a graph $G$ do not change the homology of the clique complex $\Delta(G)$, hence for any $s$-dismantlable graph $G$ the clique complex $\Delta(G)$ has trivial homology.

\begin{remark}\label{WS}
In \cite{Boulet}, the authors present the concept of $ws$-dismantlable; they say that a vertex $v$ in a graph $G$ is $ws$-dismantlable if $N_G(v)$ is dismantlable. An edge $\{u,v\}$ in the graph $G$ is called $ws$-dismantlable if $N_G(u) \cap N_G(v)$ is dismantlable. Obviously $\{E,F\}$ is not $ws$-dismantlable in $\Delta_4$. Thus, a collapse of a free pair is, in general, neither a contractible transformation nor $ws$-dismantling. Compare this remark to the proof of Lemma 4.4 of \cite{Boulet}.
\end{remark}

Despite Example \ref{example-octahedron}, we conjecture that the reverse inclusion in Theorem \ref{inclusionIC} is also true: In other words, that any strong $\II$-contractible graph is also a collapsible one. In this section, we show algorithms we have used to verify the inclusion $\CC \subset \SIC$ to several graphs. These algorithms were written in \texttt{C/C++} and are available in the repository \cite{EGC}.

To support the conjecture ($\CC \subset \SIC$), a subset of graphs of these families will be calculated, restricting the number of vertices to $n<10$.

Since connectivity is a common characteristic of these families, the strategy is  first to  obtain a collection of related graphs, then for each graph, two algorithms will be applied to determine if $G \in \SIC$ and if $G\in \CC$. The codes and the results will be found in the aforementioned repository.

\subsection{Connected graphs}

We generate the collection of connected graphs by building them up from $K_1$. Given a connected graph, we can obtain a new one by making the modifications
\begin{itemize}
\item  Add vertex $v_{n+1}$ and edge $\{v_i,v_{n+1}\}$ for some $i$ from $\{1,2,\dots,n\}$,
\item  Add edge $\{v_i,v_j\}$ to $G$ if $\{v_i,v_j\}\notin E$ with $i\neq j$.
\end{itemize}
If the resulting graph is not isomorphic to a graph already in the collection then we add it to our collection. The maximum number of vertices is determined by the computing capacity and storage for the generated graphs.

We are using a personal desktop computer, with ten vertices, which begins to require secondary memory instead of RAM for both calculations and storage. For this reason, the job is limited to nine vertices.

In the repository \cite{EGC}, codes  in \texttt{C/C++} are shared as well as some databases of adjacency matrices and geometric representations of the graphs generated. We also made available the script for the \emph{canonical labeling}, which is necessary for solving the \textit{graph isomorphism problem}. Such algorithms are also available in \cite{Mckay}, even faster optimized versions in some cases.

\subsection{Iterated construction of contractible graphs}

A recursive algorithm is then shown to determine if a family belongs to the family $\SIC$. In other words, if a graph can be reduced to a point using only the elimination operation of vertices, then the order in which you choose to delete vertices does not matter. The only graph assumed to be known as part of the family $\SIC$ is $K_1$.

The proposed algorithm assumes that there are functions or procedures to calculate the open neighborhood of a vertex $N_G(v_i)$, and to remove a vertex from the graph $G - v_i$.
\newpage

\begin{algorithm}\label{Alg:contractible.graph}
    \SetKwInOut{Input}{Input}
    \SetKwInOut{Output}{Output}

    \Input{A graph $G$ and the cardinality $n = |V(G)|$ of the vertex set.}
    \Output{The logical \texttt{TRUE} if $G \in \SIC$, or \texttt{FALSE} in otherwise.}
    \eIf{$n = 0$}
      {
        return \texttt{FALSE}\;
      } 
      {
            \eIf{$n = 1$}
      {
        return \texttt{TRUE}\;
      } {
      	\For{$i \leftarrow 1$ \KwTo $n$}{
        	\If{\textnormal{\texttt{contractible.graph}($N_G(v_i)$, $|V(N_G(v_i))|$) = \texttt{TRUE}}}{
            \Return \texttt{contractible.graph}($G- v_i, n-1$)
            }
        }
    
        }
      \Return \texttt{FALSE};
      }
	  
    \caption{\texttt{contractible.graph}}
\end{algorithm}

\subsection{The contractible reduction algorithm}
Many graphs are not in the family $\SIC$; however, it is desirable to determine how far it is possible to eliminate vertices from graph $G$. The next algorithm uses Algorithm \ref{Alg:contractible.graph} to eliminate vertices when possible.

\begin{algorithm}[H]\label{Alg:contractible.reduction}
    \SetKwInOut{Input}{Input}
    \SetKwInOut{Output}{Output}

    \Input{A finite graph $G$.}
    \Output{A reduced graph $I(G; S)$ and a ordered maximal set of vertices $S$.}

    \texttt{reduced} $\leftarrow$ \texttt{FALSE}\;
    $S \leftarrow \emptyset$\;
    \While{\textnormal{\texttt{reduced = FALSE}}}{
    	\texttt{reduced} $\leftarrow$ \texttt{TRUE} \;
    	\For{$v \in V(G)$}{
    		\If{\textnormal{\texttt{contractible.graph($N_G(v_i)$, $|V(N_G(v_i))|$) = TRUE }}}{
        	\texttt{reduced} $\leftarrow$ \texttt{FALSE}\;
			update graph $G \leftarrow G - v$\;
            $S \leftarrow S \cup \{v\}$\;
            break \textbf{for}
            }
        }
    }
    \Return $(G, S)$
    \caption{\texttt{contractible.reduction}}
\end{algorithm}

In \cite{EGC} the \texttt{C/C++} code for Algorithm \ref{Alg:contractible.reduction} can be found. The algorithm \texttt{contractible.reduction} preserves the homology groups; moreover, as we explain in the next section, the algorithm\\ \texttt{contractible.reduction} is compatible with persistent homology, i.e., the persistent homology of a family of graphs can be recovered from the sequence of the homology groups of the reduced graphs.

The algorithm \texttt{contractible.reduction} returns a maximally reduced graph with respect to deleting vertices. Reducing the graph $I(G;S)$, however, by using a contractible deleting edge is also possible as shown in Figure \ref{contractible-edge}.

\begin{figure}[h]
\begin{center}
\includegraphics[scale=0.6]{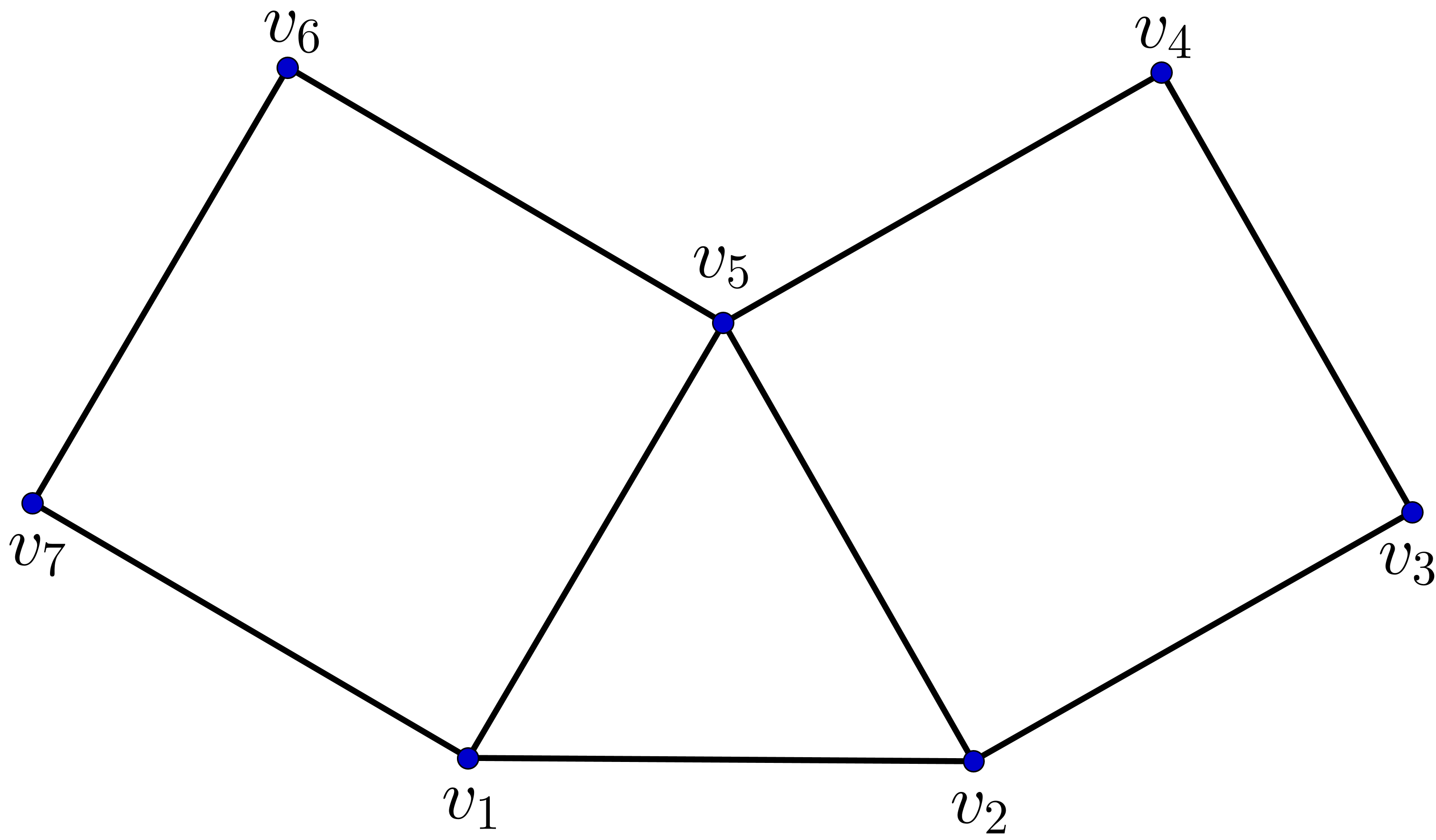}
\end{center}
\caption{Contractible-reduced graph with a contractible edge.}
\label{contractible-edge}
\end{figure}

For the graph in Figure \ref{contractible-edge}, we have $I(G;S) = G$, i.e., $S=\emptyset$. The edge $\{v_1, v_2\}$, however, can be deleted by the contractible transformation (\emph{I3}), even when it is not possible to delete any vertex of $G$ through a contractible transformation.

\begin{remark}
Families $\II_S$ and $\CC$ are infinite, but subfamilies can be calculated to make the comparison. In repository \cite{EGC} are the \texttt{C/C++} codes to calculate one subfamily limited to nine vertices and the adjacent matrix, including some geometrical representations. The algorithm to obtain the subfamily $\CC$ is by simple inspection; however, the code is provided too.
\end{remark}

\section{Homological properties of \texorpdfstring{$\II$}{I}-contractible transformations} \label{GH-CT}

Following the notation introduced in \cite{Ivashchenko:homology} and \cite{Ivashchenko:contractible}, we recall the definitions of chains, boundaries, cycles as well as induced homomorphisms and homology classes, in order to study some homological properties of the $\II$-contractible transformations and an application to computing the persistent homology of the Vietoris-Rips complex.

Let $G$ be a finite simple graph. For any complete subgraph on $n+1$ vertices $v_0, v_1, \ldots, v_n$ corresponds an $n$-simplex $\sigma^n = [v_0 v_1  \ldots v_n] \in \Delta(G)$. The oriented boundary $\partial \sigma^n$ of $\sigma^n$ is defined as the formal linear combination of its complete subgraphs on $n$ vertices:
\begin{linenomath*}
\[ \partial \sigma^n = \partial [v_0 v_1 \cdots v_n] = \sum_{k=0}^n (-1)^k [v_0 v_1 \cdots \hat{v}_k \cdots v_n].\]
\end{linenomath*}
The hat over the vertex $v_k$ means that such vertex must be omitted. An orientation for every complete subgraph is defined by restricting an arbitrary order for the vertices up to even permutations.

Let $A$ be an Abelian group; $A$ is usually taken as the integers group or a finite field. An $n$-chain of the clique complex $\Delta(G)$, with coefficients in $A$, is defined to be a formal linear combination of distinct simplices $\sigma^n$ of the clique complex:
\begin{linenomath*}
\[ c_n = \sum_{k} \alpha_k \sigma^n_k \]
\end{linenomath*}
for $\alpha_k \in A$. The addition of $n$-chains is defined in the obvious way. We then have the chain group $C_n(\Delta(G))$ of all $n$-chains in $\Delta(G)$.
Define the boundary operator $\partial : C_n(\Delta(G)) \to C_{n-1}(\Delta(G))$ by linear extension.

The boundary of a 0-chain is defined as zero. It can be proven directly that $\partial^2 = 0$, then we have a chain complex
\begin{linenomath*}
\[ \xymatrix{\cdots \ar[r]^-\partial & C_n(\Delta(G)) \ar[r]^-\partial & C_{n-1}(\Delta(G))  \ar[r]^-\partial & \cdots \ar[r]^-\partial  & C_1(\Delta(G)) \ar[r]^-\partial & C_0(\Delta(G)) \ar[r] & 0}.\]
\end{linenomath*}
 
The $n$-dimensional homology group of the graph $G$, with coefficients in $A$, is defined by
\begin{linenomath*}
\[ H_n(\Delta(G); A) = \displaystyle \frac{\ker\ (\partial : C_n(\Delta(G)) \to C_{n-1}(\Delta(G)))}{\mathrm{Im}\ (\partial : C_{n+1}(\Delta(G)) \to C_{n}(\Delta(G)))}.\]
\end{linenomath*}
The $n$-chains in the subgroup $Z_n(\Delta(G)) := \ker\ (\partial : C_n(\Delta(G)) \to C_{n-1}(\Delta(G)))$ are called $n$-cycles, and the $n$-chains in the subgroup $B_n(\Delta(G)) := \mathrm{Im}\ (\partial : C_{n+1}(\Delta(G)) \to C_{n}(\Delta(G)))$ are called $n$-boundaries. 

The homology groups of clique complexes allow us to distinguish between non-equivalent graphs, i.e., graphs with non-isomorphic homology groups of their clique complexes are not equivalent. For the graph $G$ to the left in Figure \ref{fig:1}, it can be proven that $H_0(\Delta(G); A) = A$, $H_1(\Delta(G); A) = A$ and $H_n(\Delta(G); A) = 0$ for every integer $n\geq 2$. Such groups can be interpreted in the following way: the clique complex of $G$ consists of one connected component ($H_0(\Delta(G); A) = A$), has one 1-dimensional hole ($H_1(\Delta(G); A) = A$) and no holes in higher dimensions ($H_n(\Delta(G); A) = A$, $n \geq 2$). The notion of an $n$-dimensional hole can be thought as an $n$-sphere embedded in the clique complex of $G$, assuming that every \textit{complete subgraph} $K_n$ corresponds to a $(n-1)$-dimensional filled disk.

If there is no risk of confusion, we will omit the coefficients group.

In \cite[Th. 4.9]{Ivashchenko:homology} it was proven that the contractible transformation $I: \GG \to \GG, G \mapsto I(G; v)$ does not change the homology groups of $\Delta(G)$ when $N_G(v) \in \II$, that is, 
\begin{linenomath*}
\[ H_*(\Delta(G)) \cong H_*(\Delta(I(G; v))) \cong H_*(\Delta(G-v)) . \]
\end{linenomath*}

Actually, such claim can be obtained directly from Theorem \ref{inclusionIC}, from the fact that $\Delta(G) \searrow \Delta(G-v)$ and that the homology is a homotopy invariant.

Let's see how the isomorphism $H_n(\Delta(G)) \cong H_n(\Delta(I(G; v)))$, induced by $I: \GG \to \GG, G \mapsto I(G; v)$, can be given explicitly. Let $c_n \in C_n(\Delta(G))$ be an $n$-chain, we then have that $c_n = a_n + v*b_{n-1}$ for $a_n \in \Delta(I(G;v))$ and $b_{n-1} \in \Delta(N_G(v))$. If $c_n \in Z_n(\Delta(G))$, then $\partial c_n = \partial a_n + b_{n-1} - v*\partial b_{n-1} = 0$ implies $\partial a_n + b_{n-1} = 0$ and $\partial b_{n-1} = 0$. From the last equality and the contractibility of $N_G(v)$, there exists an $n$-chain $b_n \in C_n(\Delta(N_G(v)))$ such that $\partial b_n = b_{n-1}$. We then have 
\begin{linenomath*}
\[ c_n = a_n + v*\partial b_n \ \mathrm{and}\  \partial (a_n + b_n) = 0.\]
\end{linenomath*}

Since $\partial (v*b_n) = b_n - v*\partial b_n$, the above equation can be written as 
\begin{linenomath*}
\[ c_n = (a_n + b_n) - \partial (v * b_n). \]
\end{linenomath*}

The isomorphism $I_* : H_n(\Delta(G)) \to H_n(\Delta(I(G; v)))$ is induced by the homomorphism
\begin{linenomath*}
\[ I_{\#}(v) : Z_n(\Delta(G)) \to Z_n(\Delta(I(G; v))), (a_n + b_n) - \partial (v * b_n) \mapsto a_n + b_n. \]
\end{linenomath*}

For a sequence of $\II$-contractible vertices $S = (v_{1}, v_{2}, \ldots, v_{k})$, we have a family of induced homomorphisms in cycles:
\begin{linenomath*}
\[ \xymatrix{ Z_*(\Delta(G)) \ar[r]^-{I_{\#}(v_{1})} &  Z_*(\Delta(G-\{v_{1}\})) \ar[r]^-{I_{\#}(v_{2})} & \cdots \ar[r]^-{I_{\#}(v_{k})} & Z_*(\Delta(G - \{v_{1}, v_{2}, \ldots, v_{k}\})).} \]
\end{linenomath*}
We denote the composition of such homomorphisms by $I_{\#}(S)$, i.e.,
\begin{linenomath*}
\[ I_{\#}(S) = I_{\#}(v_{1}, v_{2}, \ldots, v_{k}) =  I_{\#}(v_{k}) \circ \cdots \circ I_{\#}(v_{2}) \circ I_{\#}(v_{1}),\]
\end{linenomath*}
Define an analogous notation for the induced homomorphism in homology \[I_{*}(S) : H_*(\Delta(G)) \to H_*(\Delta(I(G;S))).\]

\begin{lemma}\label{diagram} 
Let $G_1$ be a finite graph and let $G_0 \subset G_1$ be a subgraph. If $S_0$ and $S_1$ are two sequences of $\II$-contractible vertices of $G_0$ and $G_1$, respectively, there then exists an homomorphism $\iota_* : H_*(\Delta(I(G_0; S_0))) \to H_* (\Delta(I(G_1; S_1))) $ such that the following diagram commutes:
\begin{linenomath*}
\begin{gather}
\begin{aligned}
\xymatrix{ H_* (\Delta(G_0)) \ar[r]^-{i_*} \ar[d]_-{I_*(S_0)}^-{\cong} & H_*(\Delta(G_1)) \ar[d]_-{I_*(S_1)}^-{\cong} \\ H_*(\Delta(I(G_0; S_0))) \ar@{-->}[r]^-{\iota_{*}} & H_*(\Delta(I(G_1; S_1))).}
\end{aligned}
\label{commutative-diagram}
\end{gather}
\end{linenomath*}
\end{lemma}

\begin{proof}
The iterated contractible contractions over $G_0$ and $G_1$, given by the sets $S_0$ and $S_1$, respectively, induce the following diagram:
\begin{linenomath*}
\[ \xymatrix{ Z_* (\Delta(G_0)) \ar[r]^-{I_\# (v_1)} \ar[d]_-{i_\#} & Z_*(\Delta(I(G_0; v_1)) \ar[r]^-{I_\# (v_2)} & \cdots  \ar[r]^-{I_\# (v_m)} & Z_*(\Delta(I(G_0; S_0))) \\
Z_* (\Delta(G_1)) \ar[r]^-{I_\# (w_1)} & Z_*(\Delta(I(G_1; w_1)) \ar[r]^-{I_\# (w_2)} & \cdots  \ar[r]^-{I_\# (w_n)} & Z_*(\Delta(I(G_1; S_1))) .} \]
\end{linenomath*}
The vertical arrow is induced by the inclusion $i : G_0 \hookrightarrow G_1$. From \cite{Ivashchenko:homology}, each horizontal arrow in the above diagram is an isomorphism, then $I_\#(S_0):= I_\#(v_m)\circ \cdots \circ I_\#(v_1)$ and $I_\#(S_1):= I_\#(w_m)\circ \cdots \circ I_\#(w_1)$ are also isomorphisms. Let $\iota : Z_*(\Delta(I(G_0; S_0))) \to Z_*(\Delta((G_1; S_1))$ be defined by $z \mapsto I_{\#}(S_1) \circ i_\# \circ I_{\#}(S_0)^{-1}(z)$. Thus, we have the following two commutative diagrams:
\vspace*{-4mm}
\begin{multicols}{2}
\[ \xymatrix{ Z_* (\Delta(G_0)) \ar[r]^-{I_\# (v_1)} \ar[d]_-{i_\#} & Z_*(\Delta(I(G_0; S_0)))  \ar@{-->}[d]^-{\iota} \\ Z_* (\Delta(G_1)) \ar[r]^-{I_\# (w_1)} & Z_*(\Delta(I(G_1; S_1)))} \]

\[ \xymatrix{ B_* (\Delta(G_0)) \ar[r]^-{I_\# (v_1)} \ar[d]_-{i_\#} & B_*(\Delta(I(G_0; S_0)))  \ar@{-->}[d]^-{\iota} \\ B_* (\Delta(G_1)) \ar[r]^-{I_\# (w_1)} & B_*(\Delta(I(G_1; S_1))).} \]
\end{multicols}
The first diagram from left to right is commutative by construction, and the second one is commutative by restriction to the boundaries' subgroup. Taking the induced morphisms in the quotient groups (homology groups), we obtain the commutative diagram (\ref{commutative-diagram}).
\end{proof}

\subsection{Application: Persistent homology of the Vietoris-Rips complex} \label{Appli} 

In some applications of algebraic topology, such as topological data analysis (TDA), the \textit{shape of a data cloud} is studied through the persistent homology of a filtered simplicial structure (cf. \cite{Carlsson2009,Patania2017}). Choosing the appropriate simplicial structure and the parameter or variable used to construct the filtration are key steps in the topological analysis. A standard simplicial structure which defines a geometry over a data cloud is known as the Vietoris-Rips complex.

Let $N$ be a finite point cloud in some metric space $(M, \mathrm{d})$, and let $\varepsilon$ be a nonnegative real number. The Vietoris-Rips complex $\mathrm{VR}(N; \varepsilon)$ is the abstract simplicial complex with $N$ as set of vertices, and $\sigma \subset N$ is a simplex if, and only if, the distance between any two points in $\sigma$ is less than or equal to $\varepsilon$. Clearly, $\mathrm{VR}(N; \varepsilon) \subset \mathrm{VR}(N; \varepsilon')$ for all $\varepsilon \leq \varepsilon'$. In addition, sorting all values $\mathrm{d}(u, v)$ for any $u, v \in N$, let's say $\{0 = \varepsilon_0 < \varepsilon_1 < \ldots < \varepsilon_m\}$, we have the Vietoris-Rips filtration
\begin{linenomath*}
\[ \mathrm{VR}(N; \varepsilon_0) \subset \mathrm{VR}(N; \varepsilon_1) \subset \cdots \subset \mathrm{VR}(N; \varepsilon_m). \]
\end{linenomath*}

We can denote $\mathrm{VR}(N; \varepsilon_i)$ simply by $\mathrm{VR}_i$ when there is no risk of confusion about the point cloud or the filtration. In \cite{Zomorodian2010} several algorithms can be found for the construction of this simplicial structure.

Since the shape of a point cloud can be summarized, in a certain sense, by the persistent homology of its (filtered) Vietoris-Rips complex, such construction is really useful in applications.

Given a finite sequence of simplicial complexes $\Delta_0 \subset \Delta_1 \subset \cdots \subset \Delta_m$, for $i, j \in \{0, 1, \ldots, m \}$ such that $i \leq j$, the \emph{$(i, j)$-persistent $p$-homology group} $H_p^{i, j}$ of the filtration is defined as $\mathrm{Im}( H_p(\Delta_i) \to H_p(\Delta_j))$. The generators of $H_p^{i, j}$ are those ``holes'' that survive from $\Delta_i$ to $\Delta_j$. See \cite{Zomorodian2005} and \cite{Ghrist:2007} for a deeper analysis.

\begin{theorem}\label{persistence-preserved}
Let $\mathrm{VR}_0 \subset \mathrm{VR}_1 \subset \cdots \subset \mathrm{VR}_m$ be the Vietoris-Rips filtration of any finite point cloud, and let $H_p^{i,j}$ be the $(i,j)$-persistent $p$-homology group of the filtration.  If $G_i := \mathrm{VR}_i^{(1)}$ is given by the 1-skeleton of the complex $\mathrm{VR}_i$ and $S_i$ is a sequence of $\II_S$-contractible vertices of $G_i$, for all $1 \leq i \leq m$, then
\begin{linenomath*}
\begin{equation}\label{iso-persistent-homology}
H_p^{i,j} \cong \mathrm{Im}(\iota_* : H_p(\Delta(I(G_i; S_i))) \to H_p(\Delta(I(G_j; S_j)))).
\end{equation}
\end{linenomath*}
\end{theorem}

\begin{proof}
By definition $H_p(\mathrm{VR}_i) \cong H_p(\Delta(G_i))$ and from Lemma \ref{diagram}, we have the commutative diagram
\begin{linenomath*}
\[ \xymatrix{ H_p (\mathrm{VR}_i) \cong H_p(\Delta(G_i)) \ar[r] \ar[d]_-{I_*(S_i)}^-{\cong} & H_p (\mathrm{VR}_{i+1})  \cong H_p(\Delta(G_{i+1})) \ar[d]_-{I_*(S_{i+1})}^-{\cong} \\ H_p (\Delta(I(G_i; S_i))) \ar[r]^-{\iota_{*}} & H_p (\Delta(I(G_{i+1}; S_{i+1}))).} \]
\end{linenomath*}
The isomorphism (\ref{iso-persistent-homology}), for any $j \geq 1$, is a consequence of the functoriality of the homology theory and the Persistence Equivalence Theorem \cite[Sec. VII.2]{Edelsbrunner:08}.
\end{proof}

Theorem \ref{persistence-preserved} provides an alternative method for computing the persistent homology of the Vietoris-Rips filtration $\{\mathrm{VR}_i\}_{i=1}^m$ through the sequence of the homology groups of the simplicial complexes $\{\Delta(I(\mathrm{VR}_i^{(1)}; S_i))\}_{i=1}^m$. The remarkable fact about this theorem is that even if $\{\Delta(I(G_i; S_i))\}_{i=1}^m$ is not a simplicial filtration (which is the general case), there is an induced homomorphism at the level of the homology groups that allows us to recover the persistent homology of $\{\mathrm{VR}_i\}_{i=1}^m$.

As we pointed out before, the Vietoris-Rips filtration is frequently used in many applications of topological data analysis, mainly because is enough to consider the distance matrix $D_N$ of a point cloud $N$ under study, to construct such filtration, or equivalently, to consider the (filtered) family of graphs defined by $D_N$ and take then the clique complex of each one. However, the computational cost of this approach can be really expensive even for ``small'' point clouds, since for larger values of the proximity parameter for a point cloud on $n$ points, we could get up to $2^n - 1$ simplices. 

Our approach proposes a pre-processing of the data cloud at each level of the filtration for then to proceed to compute the homology groups as well as the induced homomorphisms. Such approach has been successfully applied in the software Perseus \cite{PERSEUS}, which performs certain homology-preserving Morse theoretic reductions on several combinatorial structures including the simplicial one. Just like for contractible transformations, there are some transformations in discrete Morse theory that preserve the homology groups (see \cite{Forman:1998} and \cite{Nanda:2013}). Another work with this approach can be found in \cite{Dlotko:14}. There are actually many libraries and software to compute persistent homology from several approaches, see \cite{Otter2017} for an extensive list. One of these programs is the software Ripser (cf. \cite{Ripser}), a very efficient software for computing the persistent homology of the Vietoris-Rips complex. 

We conclude this section with an example of the persistent homology of the Vietoris-Rips complex of a point cloud in the plane.

\begin{example}

Let $N=\{v_1, \ldots, v_6\} \subset \mathbb{R}^2$ be a point cloud, and let 
\begin{linenomath*}
\begin{equation} \label{Example:filtration}
    \mathrm{VR}(N; \varepsilon_1 =0 ) \subset \mathrm{VR}(N; \varepsilon_2 = 1.5 ) \subset \mathrm{VR}(N; \varepsilon_3 = 2.1 ) \subset \mathrm{VR}(N; \varepsilon_4 = 2.6 ) \subset \mathrm{VR}(N; \varepsilon_5 = 2.7 ) 
\end{equation}
\end{linenomath*}
be the Vietoris-Rips filtration, as shown in the next picture.

\begin{figure}[htbp]
\begin{minipage}[t]{0.19\linewidth}
    \includegraphics[width=\linewidth]{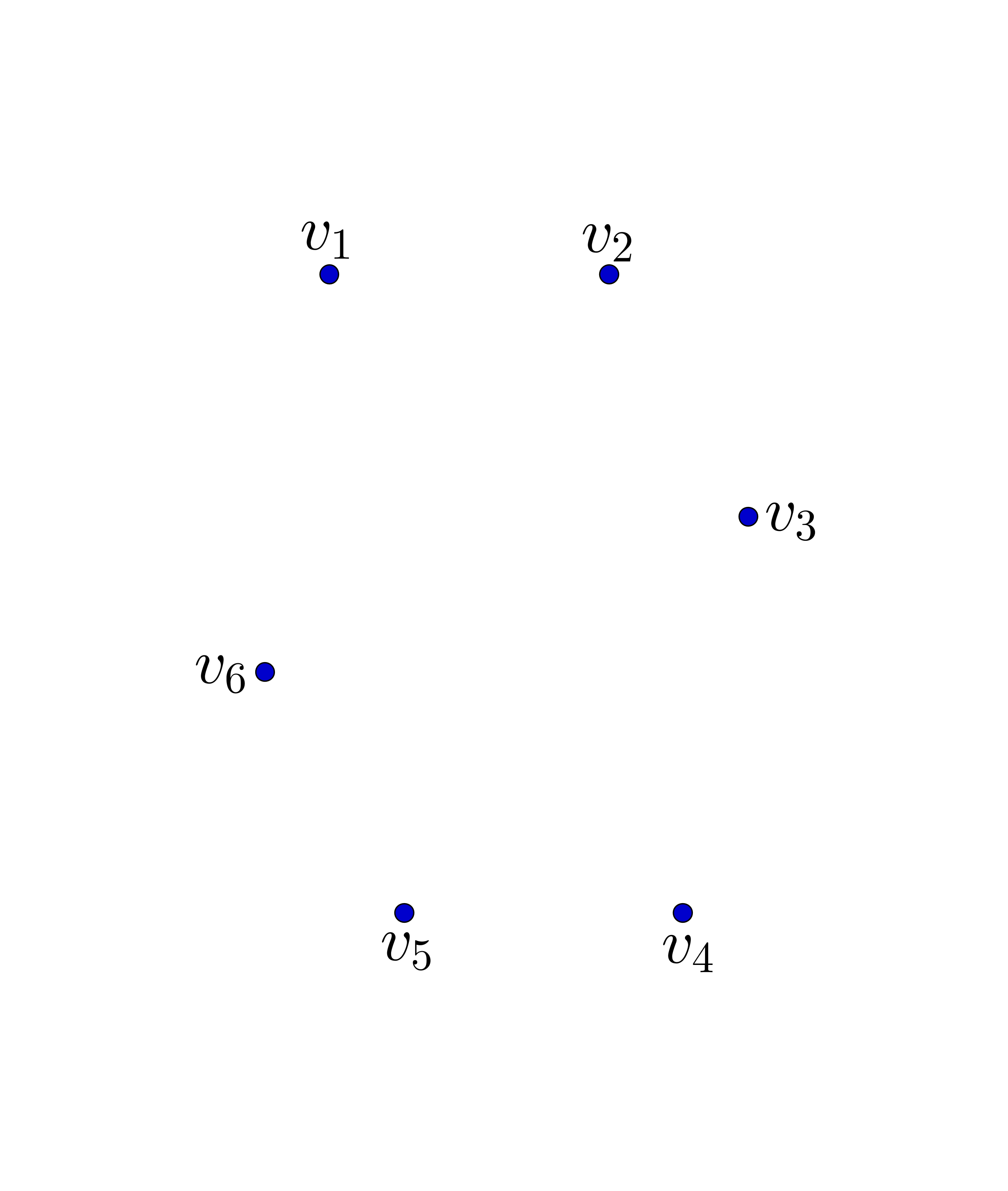}
\end{minipage}%
    \hfill%
\begin{minipage}[t]{0.19\linewidth}
    \includegraphics[width=\linewidth]{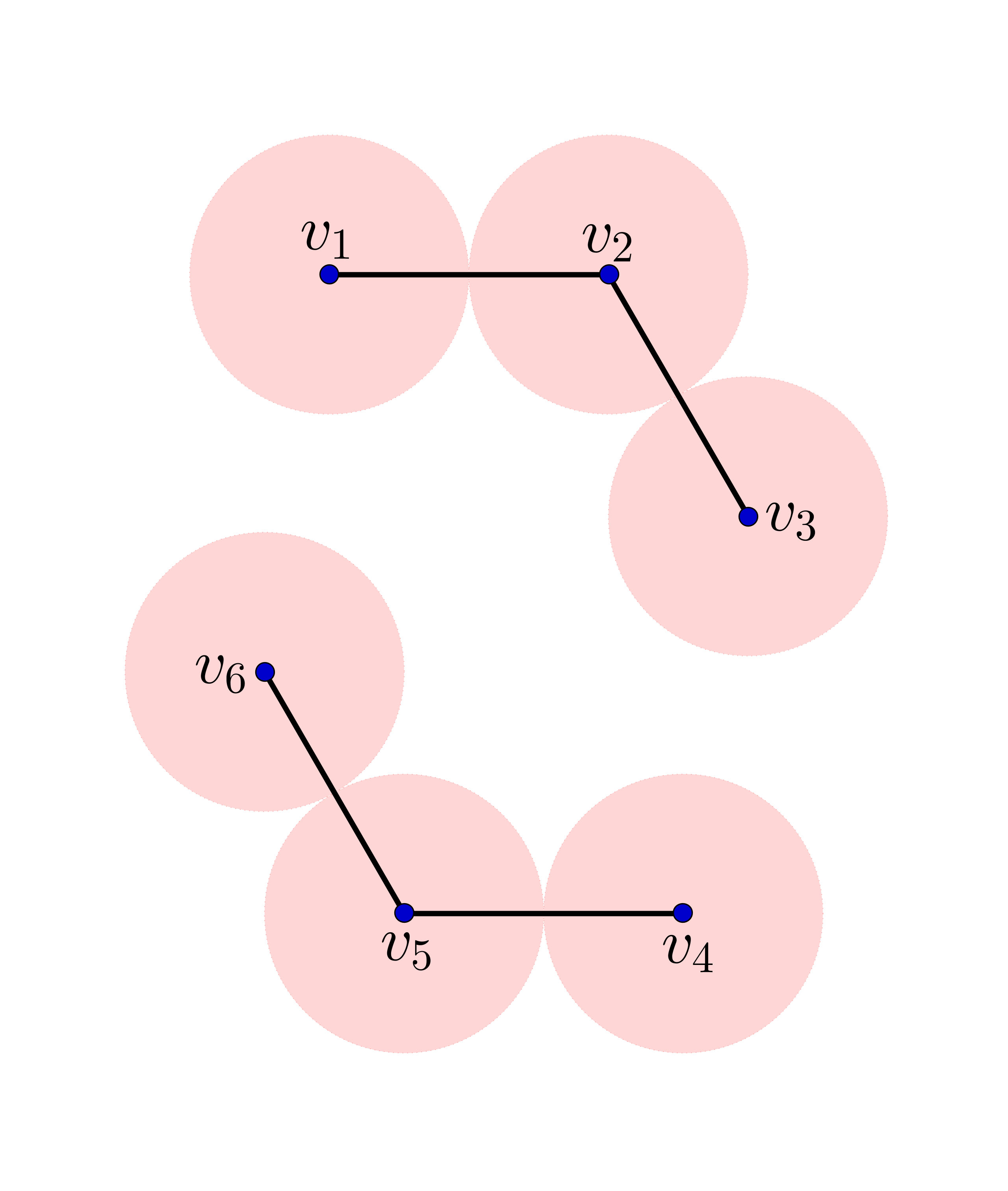}
\end{minipage} 
    \hfill%
\begin{minipage}[t]{0.19\linewidth}
    \includegraphics[width=\linewidth]{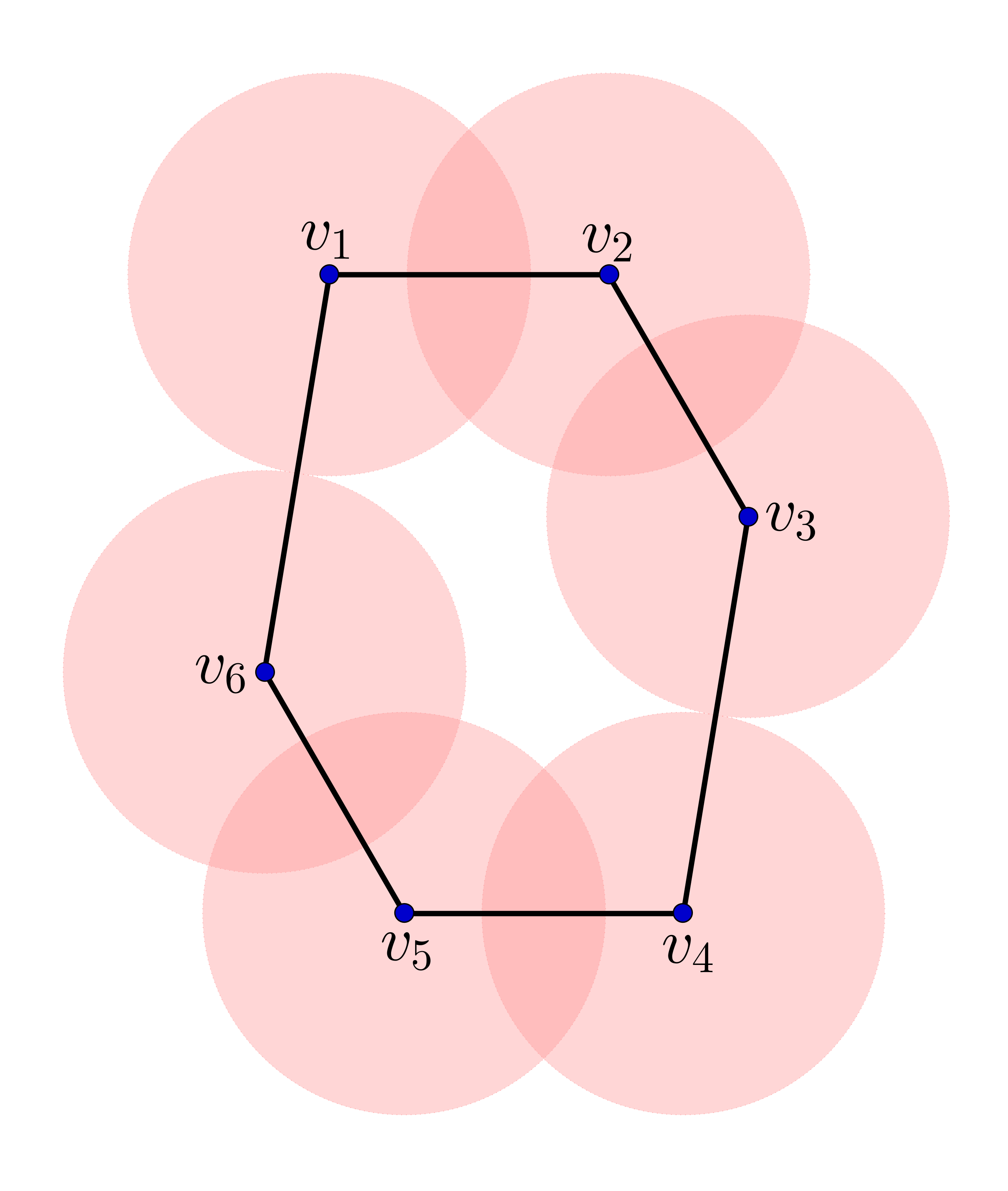}
\end{minipage} 
    \hfill%
\begin{minipage}[t]{0.19\linewidth}
    \includegraphics[width=\linewidth]{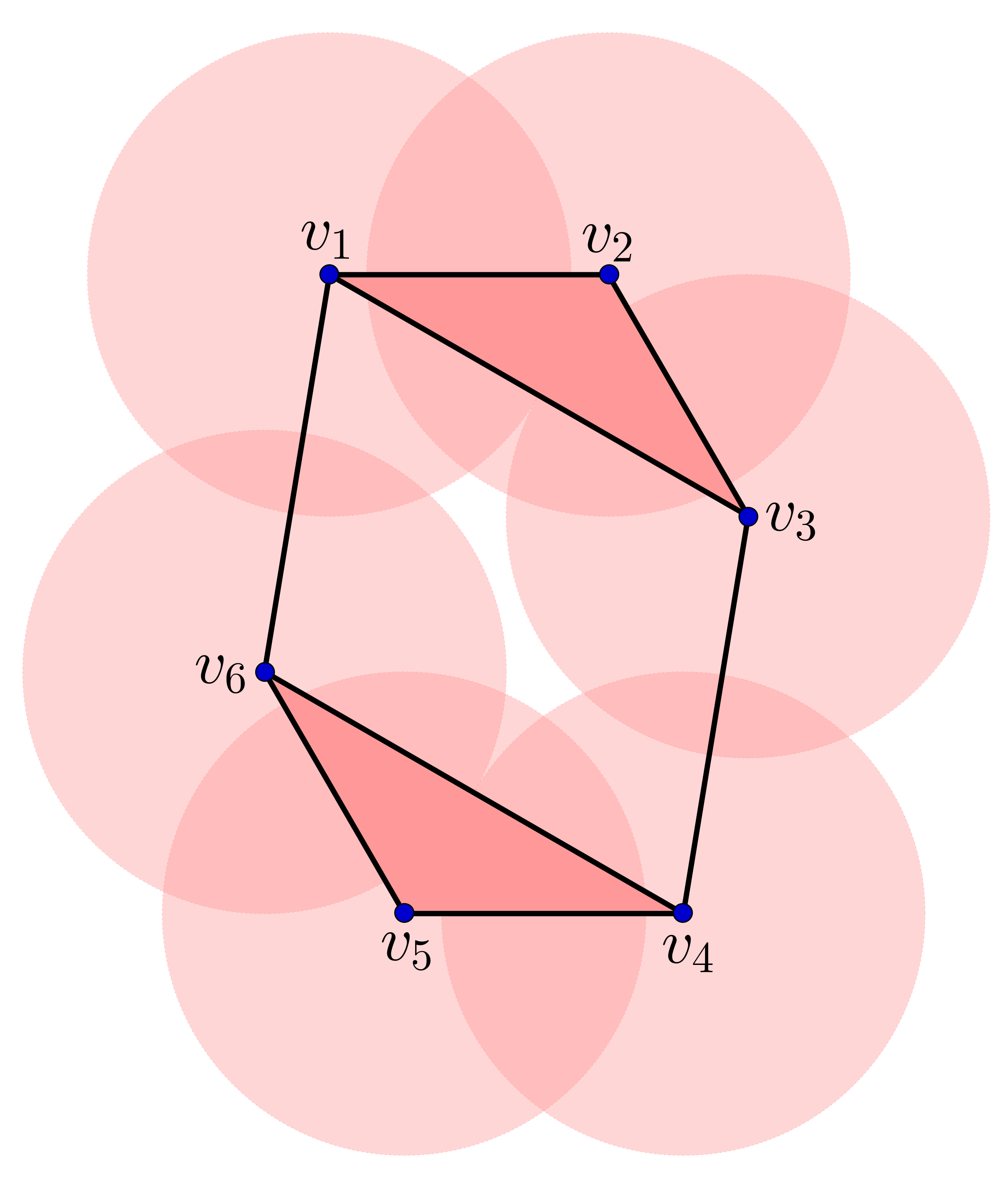}
\end{minipage} 
    \hfill%
\begin{minipage}[t]{0.19\linewidth}
    \includegraphics[width=\linewidth]{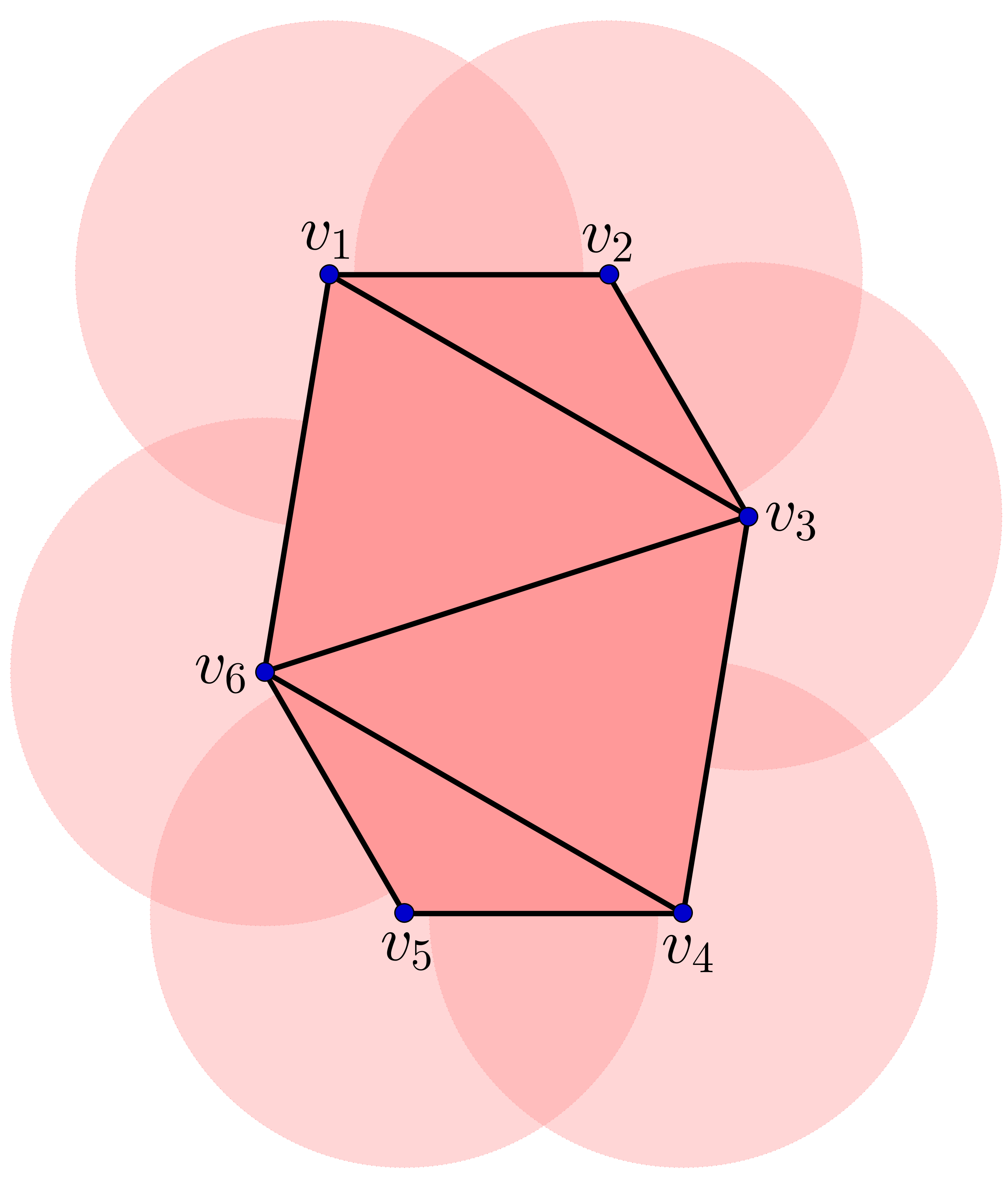}
\end{minipage} 
\end{figure}

The following picture shows the sequence of graphs given by the \texttt{contractible.reduction} algorithm over the 1-skeleton of each simplicial complex, $G_i := I(\mathrm{VR}(N; \varepsilon_i)^{(1)})$.

\begin{figure}[htbp]
\begin{minipage}[t]{0.19\linewidth}
    \includegraphics[width=\linewidth]{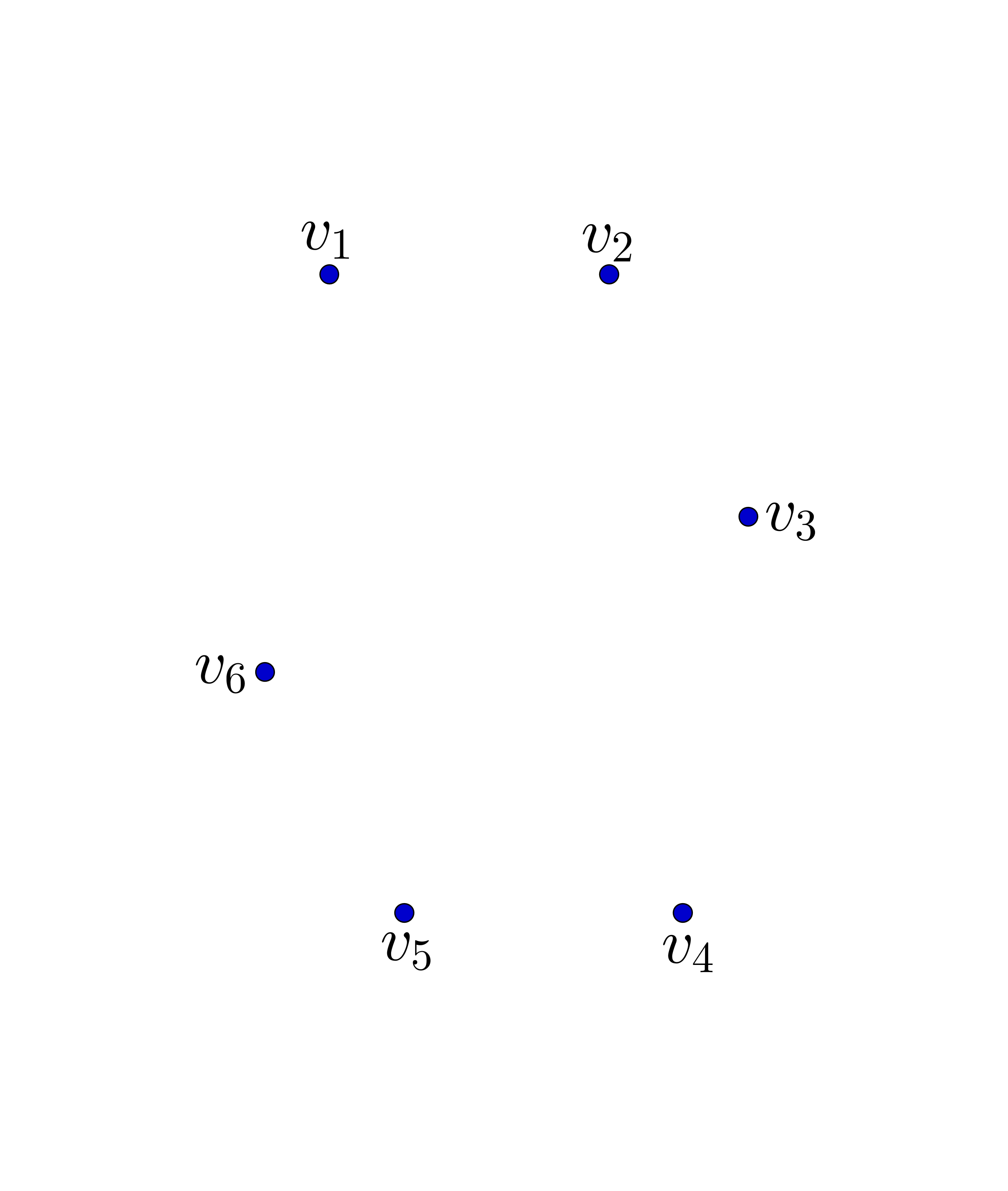}
\end{minipage}%
    \hfill%
\begin{minipage}[t]{0.19\linewidth}
    \includegraphics[width=\linewidth]{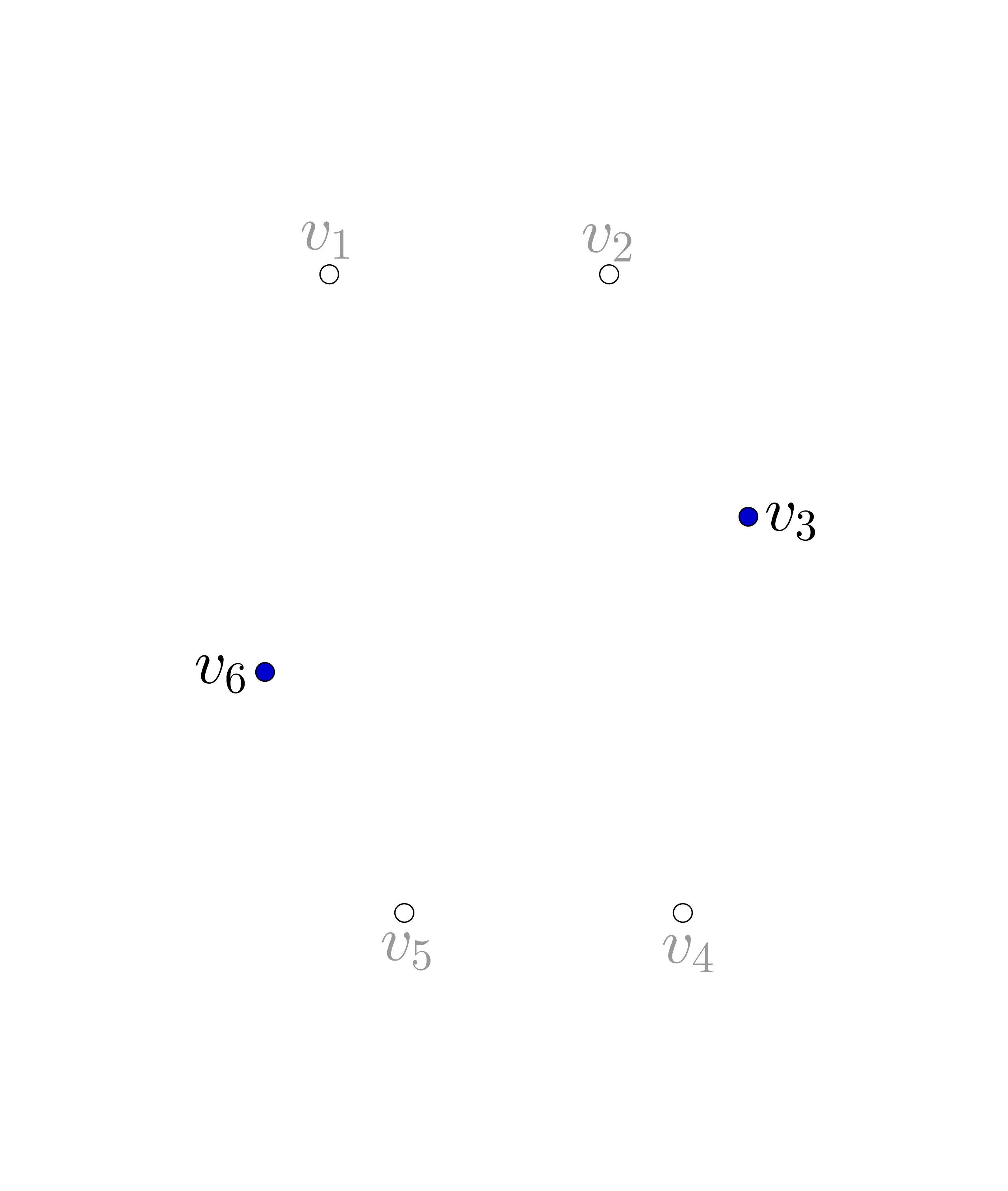}
\end{minipage} 
    \hfill%
\begin{minipage}[t]{0.19\linewidth}
    \includegraphics[width=\linewidth]{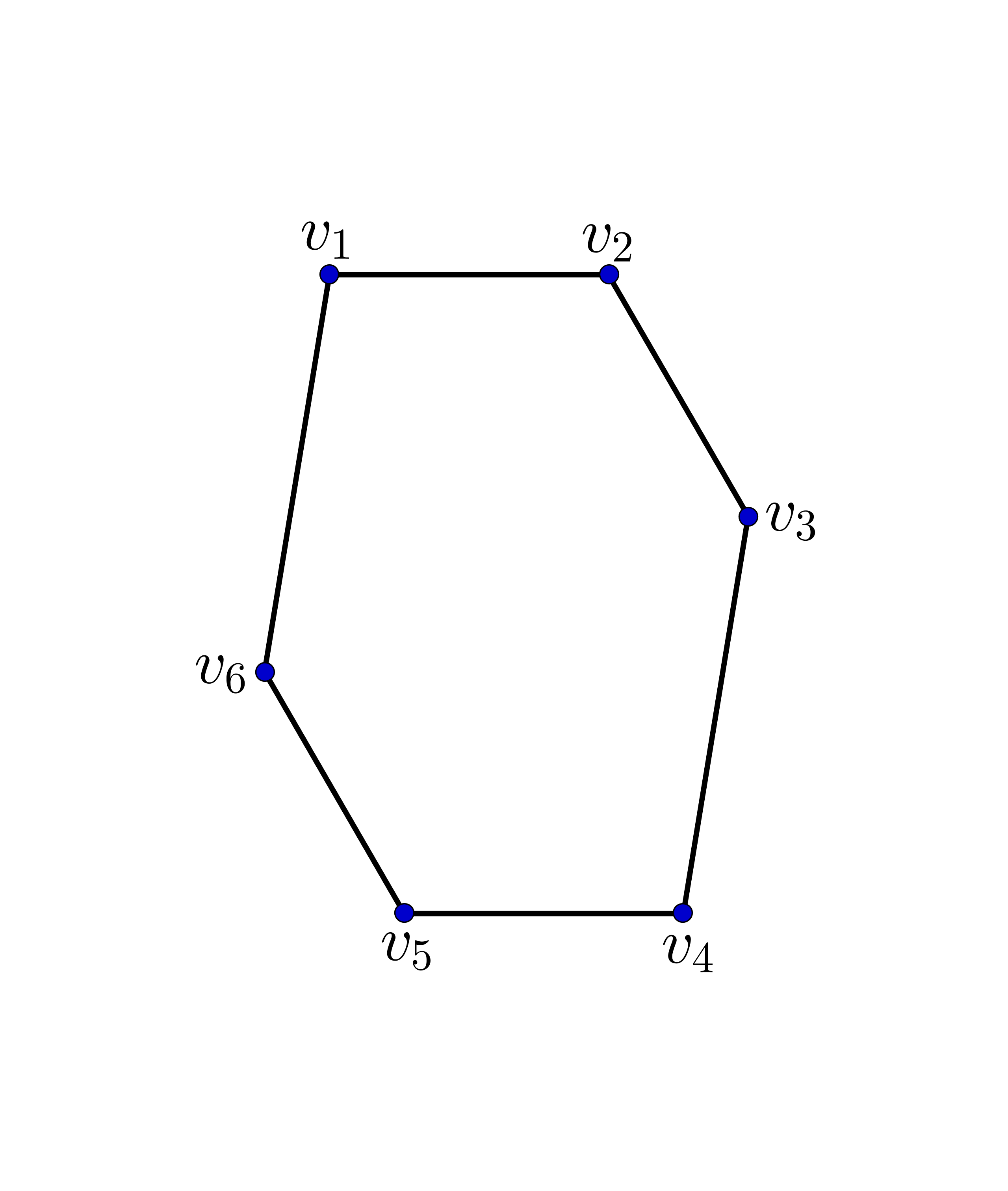}
\end{minipage} 
    \hfill%
\begin{minipage}[t]{0.19\linewidth}
    \includegraphics[width=\linewidth]{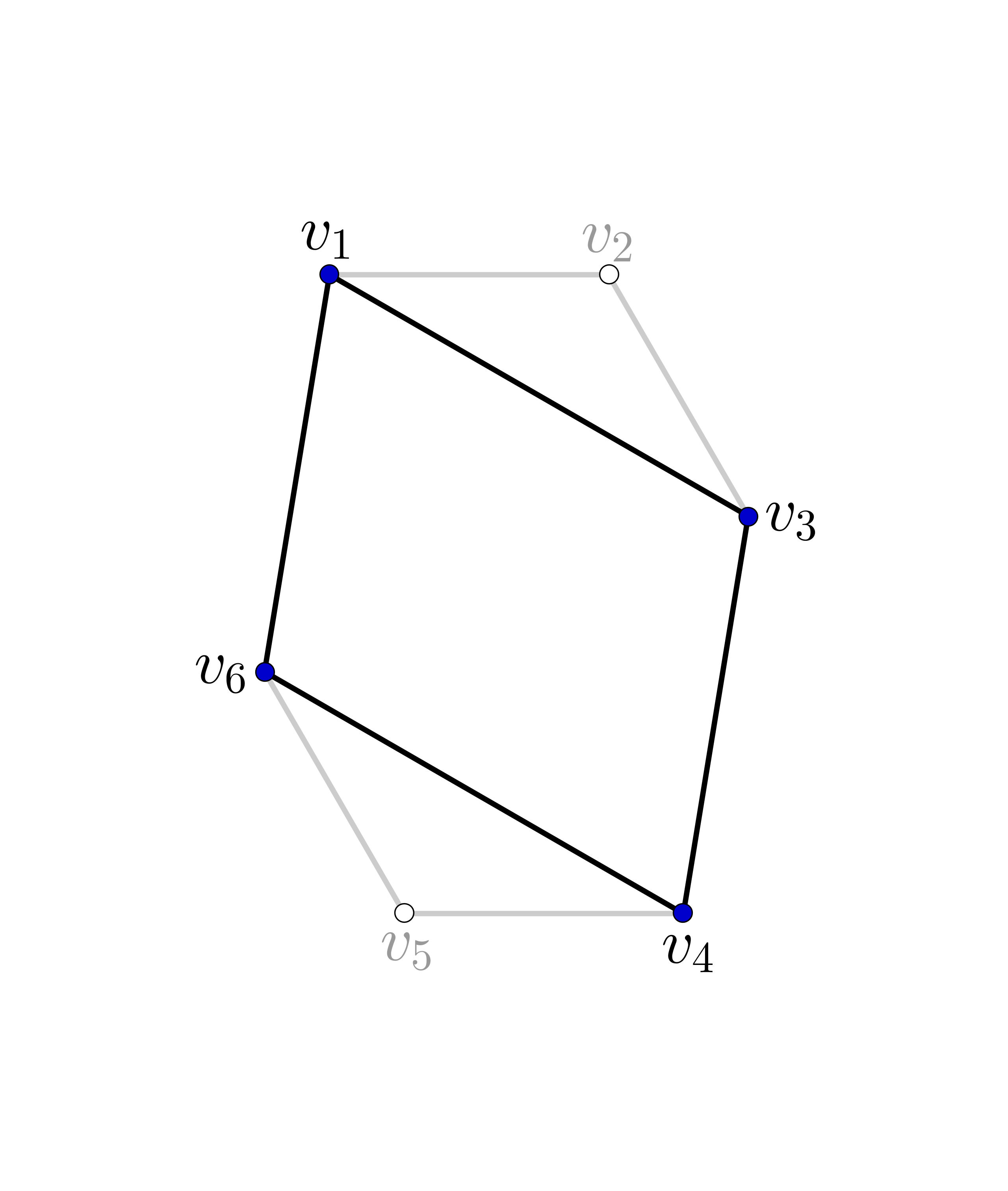}
\end{minipage} 
    \hfill%
\begin{minipage}[t]{0.19\linewidth}
    \includegraphics[width=\linewidth]{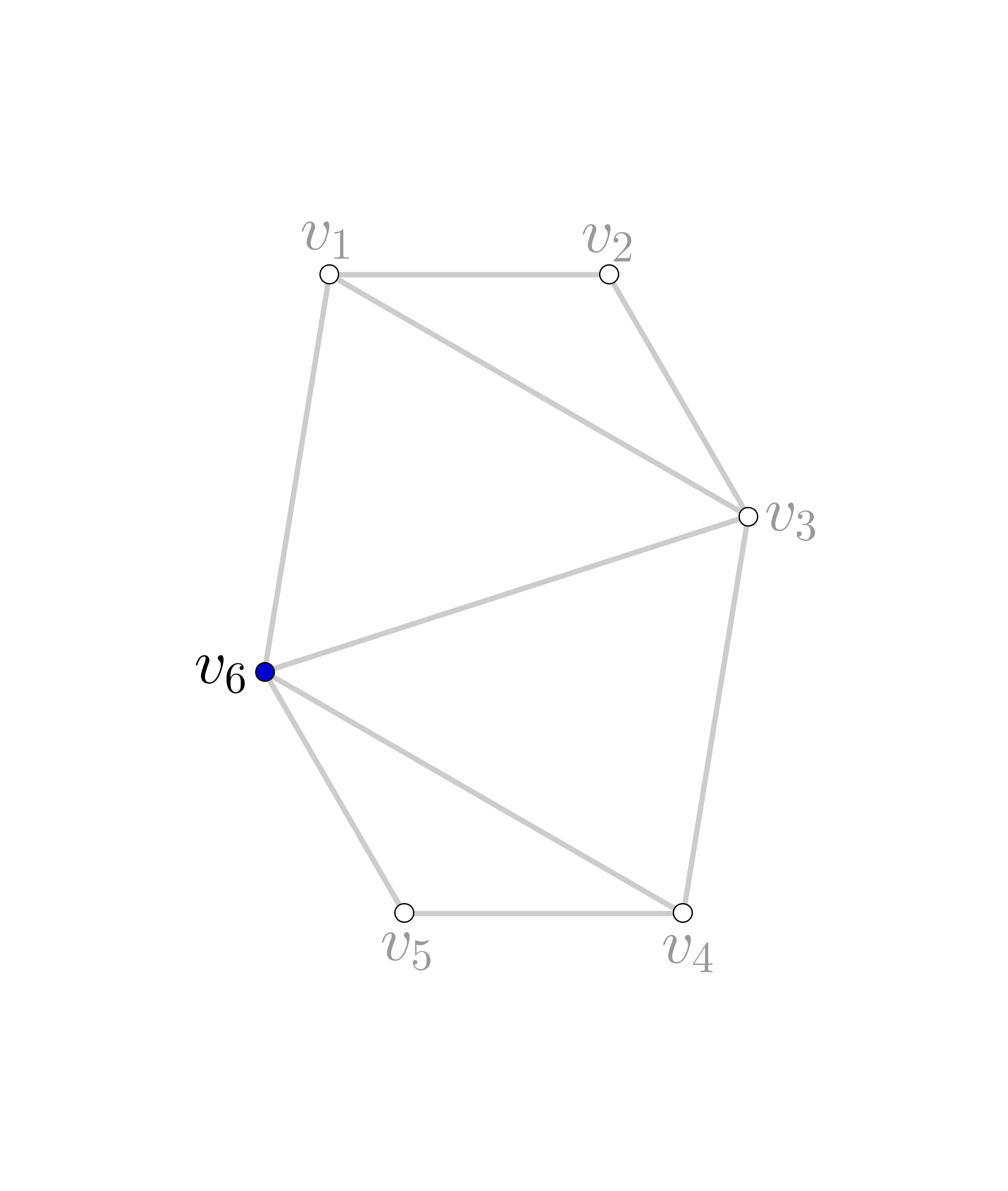}
\end{minipage} 
\end{figure}

Clearly, the sequence $\{G_i\}$ does not have a filtration structure. The following table shows the generators of the $p$-homology groups ($p=0,1$) of the clique complex for each reduced graph $G_i$.
\bigskip

\begin{center}
\begin{tabular}{cccccc}
    & $H_p(\Delta(G_1))$    & $H_p(\Delta(G_2))$    & $H_p(\Delta(G_3))$    & $H_p(\Delta(G_4))$    & $H_p(\Delta(G_5))$    \\ \hline
    \multicolumn{1}{|c|}{\multirow{6}{*}{$p=0$}} & \multicolumn{1}{c|}{$[v_1]$} &
    \multicolumn{1}{c|}{\multirow{3}{*}{$[v_3]$}} & \multicolumn{1}{c|}{\multirow{6}{*}{$[v_1]$}} &
    \multicolumn{1}{c|}{\multirow{6}{*}{$[v_1]$}} & \multicolumn{1}{c|}{\multirow{6}{*}{$[v_6]$}} \\ \cline{2-2}
    \multicolumn{1}{|c|}{} & \multicolumn{1}{c|}{$[v_2]$} &
    \multicolumn{1}{c|}{} & \multicolumn{1}{c|}{} &
    \multicolumn{1}{c|}{} & \multicolumn{1}{c|}{} \\ \cline{2-2}
    \multicolumn{1}{|c|}{} & \multicolumn{1}{c|}{$[v_3]$} & 
    \multicolumn{1}{c|}{} & \multicolumn{1}{c|}{} &
    \multicolumn{1}{c|}{} & \multicolumn{1}{c|}{} \\ \cline{2-3}
    \multicolumn{1}{|c|}{} & \multicolumn{1}{c|}{$[v_4]$} & 
    \multicolumn{1}{c|}{\multirow{3}{*}{$[v_6]$}} & \multicolumn{1}{c|}{} &
    \multicolumn{1}{c|}{} & \multicolumn{1}{c|}{} \\ \cline{2-2}
    \multicolumn{1}{|c|}{} & \multicolumn{1}{c|}{$[v_5]$} & 
    \multicolumn{1}{c|}{} & \multicolumn{1}{c|}{} &
    \multicolumn{1}{c|}{} & \multicolumn{1}{c|}{} \\ \cline{2-2}
    \multicolumn{1}{|c|}{} & \multicolumn{1}{c|}{$[v_6]$} &
    \multicolumn{1}{c|}{} & \multicolumn{1}{c|}{} & 
    \multicolumn{1}{c|}{} & \multicolumn{1}{c|}{} \\ \hline
    \multicolumn{1}{|c|}{$p=1$} & \multicolumn{1}{c|}{0} &
    \multicolumn{1}{c|}{0} & \multicolumn{1}{c|}{$\alpha$} &
    \multicolumn{1}{c|}{$\beta$} & \multicolumn{1}{c|}{0} \\ \hline
\end{tabular}
\end{center}

\noindent We are denoting $\alpha = [v_1v_2]+[v_2 v_3]+[v_3 v_4] + [v_4 v_5]+[v_5 v_6] + [v_1 v_6]$ and $\beta = [v_1 v_3] + [v_3 v_4] + [v_4 v_6] + [v_1 v_6]$. Such table of the persistent homology classes is usually represented graphically as a barcode or as a persistence diagram (see \cite{Ghrist:2007}).

\begin{figure}[hbt]
    \centering
    \includegraphics[width=0.4\textwidth]{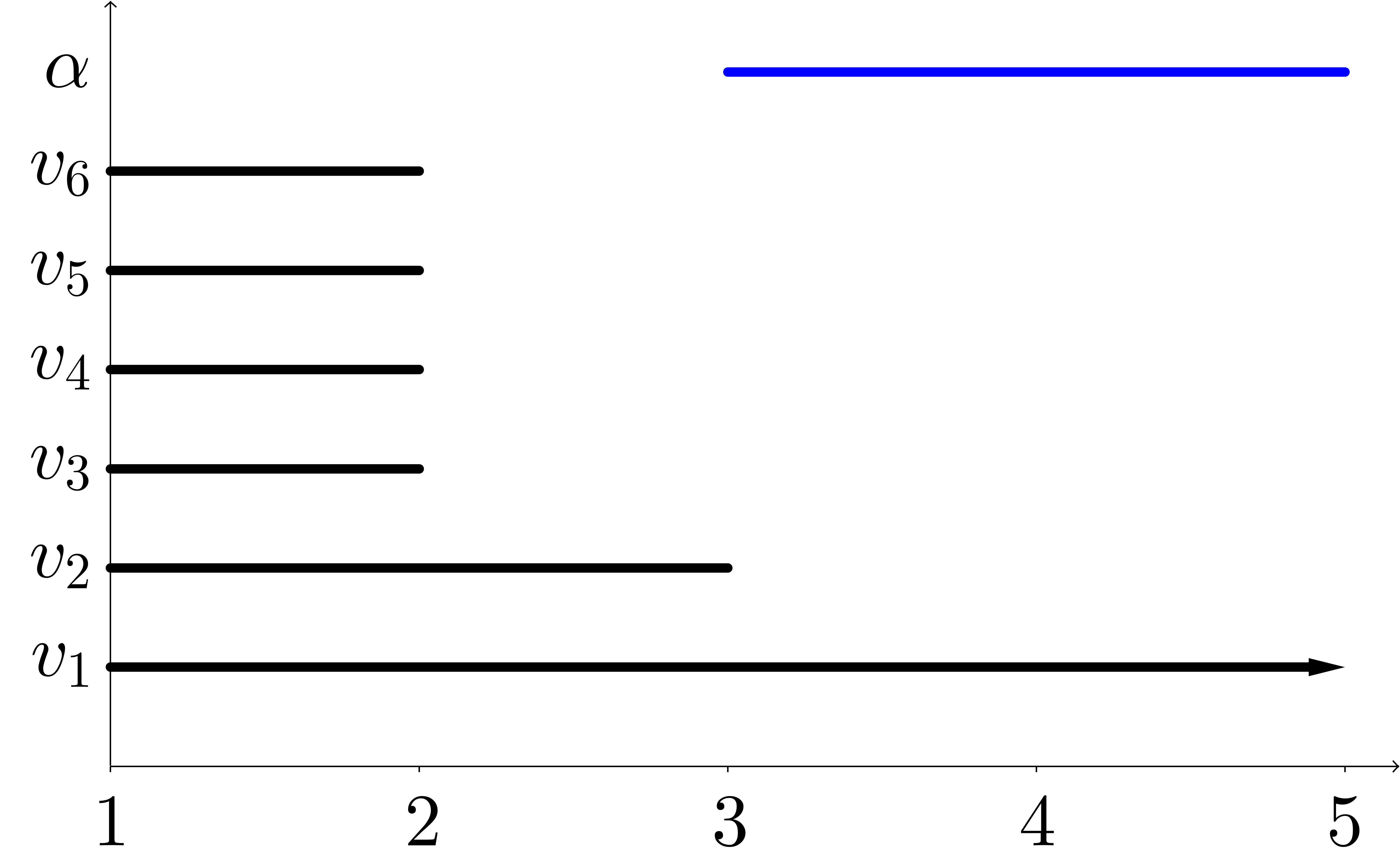}
    \caption{Barcode for the persistent homology of the Vietoris-Rips filtration (\ref{Example:filtration}).}
\end{figure}

\end{example}

At the repository in \cite{EGC}, some visual material about the construction of the Vietoris-Rips complex (as a graph) can be found next to the corresponding contractible reduced structure. The experiments show a behavior similar to the above example.

\bibliographystyle{siam}
\bibliography{BSMM_bibliography}
\end{document}